\documentclass{amsart}   %%%%%%%Please do not change
\usepackage{ulem}
\usepackage{graphicx, amsmath, amssymb, amsthm, amscd,pstricks-add}
\usepackage{enumerate}
\usepackage{color}
\usepackage{epsf}
\usepackage{pgfplots}
\pgfplotsset{compat=1.15}
\usepackage{mathrsfs}
\usetikzlibrary{arrows}
\definecolor{ffqqqq}{rgb}{1,0,0}
\usepackage{url}
\usepackage[main=english,polish]{babel}
\usepackage[T1]{fontenc}
\usepackage{hyperref}
\usepackage{subcaption}

\newtheorem{thm}{Theorem}[section]
\newtheorem{cor}[thm]{Corollary}
\newtheorem{lem}[thm]{Lemma}
\newtheorem{prop}[thm]{Proposition}

\theoremstyle{definition}
\newtheorem{defn}[thm]{Definition}
\newtheorem{rem}[thm]{Remark}
\newtheorem{exmp}[thm]{Example}
\newtheoremstyle{named}{}{}{\itshape}{}{\bfseries}{.}{.5em}{\thmnote{#3}#1}
\theoremstyle{named}

\newcommand{\ClassW}{\text{Class}(W)}
\newcommand{\N}{\mathbb{N}}

\newcommand{\Z}{\mathbb{Z}}
\newcommand{\Q}{\mathbb{Q}}
\newcommand{\R}{\mathbb{R}}

\newcommand{\eps}{\varepsilon}
\newcommand{\ra}{\rightarrow}

\DeclareMathOperator{\diam}{diam}
\DeclareMathOperator{\dist}{dist}

\DeclareMathOperator{\Br}{Br}
\DeclareMathOperator{\End}{End}

\DeclareMathOperator{\Inter}{int}
\DeclareMathOperator{\Fix}{Fix}

\numberwithin{equation}{section}
\begin{document}

	\title{Tranched graphs: consequences for topology and dynamics}

	\author[M.\ Kowalewski]{Micha\l{}~Kowalewski}
	\author[P.\ Oprocha]{Piotr Oprocha}

\address[M. Kowalewski]{AGH University of Krakow, Faculty of Applied Mathematics,
al.\ Mickiewicza 30, 30-059 Krak\'ow, Poland.}
\email{kowalewski@agh.edu.pl}

 \address[P.\ Oprocha]{Centre of Excellence IT4Innovations - Institute for Research and Applications of Fuzzy Modeling, University of Ostrava, 30. dubna 22, 701 03 Ostrava 1, Czech Republic  -- $\&$ --
 AGH University of Krakow, Faculty of Applied Mathematics,
al.\ Mickiewicza 30, 30-059 Krak\'ow, Poland. }
\email{piotr.oprocha@osu.cz}

        \begin{abstract}
		We compare quasi-graphs and generalized $\sin(1/x)$-type continua, which are two classes of continua that generalize topological graphs and contain the Warsaw circle as a nontrivial common element.   We show that neither class is a subset of the other, provide some characterizations, and present illustrative examples. We unify both approaches by considering the class of \textit{tranched graphs},  compare it to concepts known from the literature, and describe how the topological structure of its elements restricts possible dynamics.
	\end{abstract}
 
	\maketitle

\section{Introduction}
In the paper, we study the relationship between \textit{quasi-graphs} and \textit{generalized sin(1/x)-type continua}. The two classes were defined independently, non-trivially extending the class of topological graphs. Let us present a brief, informal description of these classes, while for a formal definition, the reader is referred to Section~\ref{sec:2}.
If we view topological graphs as arcwise connected unions of arcs, 
then roughly speaking, by an analogy, we can view quasi-graphs as arcwise connected unions of arcs and quasi-arcs. The definition of generalized sin(1/x)-type continua is based on an analogy to objects that generalize the unit interval, the so-called type-$\lambda$ continua, defined and studied by Kuratowski (see \cite[\S 48, Ch. III, footnote on p.197]{MR259835}) and subsequent mathematicians (e.g. see \cite{MR1073777,MR701520})). In this approach, a topological space generalizes a topological graph if we can define a monotone map onto a topological graph that is 1-1 on a sufficiently large set of points (with additional assumptions). 
Simple examples suggest that the two definitions may be equivalent under some mild assumptions. For instance, the Warsaw circle (see: Figure~\ref{fig:warsaw cricle}) is both a quasi-graph and a generalized sin(1/x)-type continuum. Motivated by this, we set as the main goal of this paper a study on the relationship between classes of quasi-graphs and generalized sin(1/x)-type continua.

    \begin{figure}[ht]

            \centering
            \includegraphics[scale=0.28]{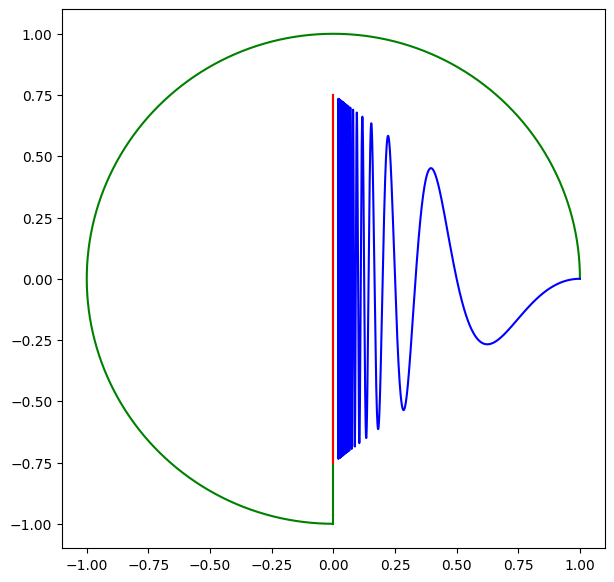}
            \includegraphics[scale=0.28]{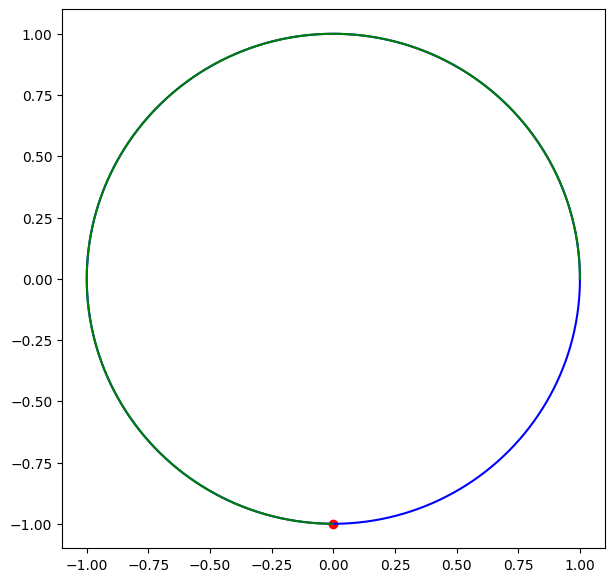}
            \caption{The Warsaw circle $W$ and its image $\phi(W)$ under mapping $\phi$ from Definition   \ref{def_sin1/x}. The points in topological graph $\phi(W)$ are colored in accordance to their preimage. Oscillatory quasi-arc required by Definition~\ref{def_quasigraph} is marked in blue.}
            \label{fig:warsaw cricle}
        \end{figure}

One can easily find an example  of a generalized sin(1/x)-type continuum that is not a quasi-graph, since the definition of generalized sin(1/x)-type continuum does not demand arcwise connectedness of the space; however, adding arcwise connectedness to the definition still does not lead to equivalence of definitions. The results of the paper include illustrative examples that present the differences between quasi-graphs and generalized sin(1/x)-type continua (see Section~\ref{sec:3} and Section~\ref{sec:4} for more details). We introduce the class of \textit{tranched graphs}, that contains all quasi-graphs and all generalized sin(1/x)-type continua, but is still small enough to provide some concrete characterizations. We also characterize both classes relative to each other, in particular, we try to understand the intersection of these classes. For the latter, we find very useful the class of continua known in literature as $\ClassW{}$, defined by Lelek in 1972 (see \cite{GT} and the summary of classical results in the book of Illanes and Nadler \cite[Ch. VIII \& IX]{Hyperspaces}).

The main results of the paper can be summarized as follows.
In Theorems~\ref{when_quasi_is_sin} and~\ref{thm:commonbase} we show that for a quasi-graph $X$, a sufficient condition to be a generalized sin(1/x)-type continuum is that for any connected component $\Lambda$ of the union of limit sets of oscillatory quasi-arcs in $X$: 
\begin{enumerate}
    \item there is a quasi-arc $L \subset X$ such that $\omega(L)=\Lambda$, and
    \item $\Lambda \in \ClassW{}$ 
\end{enumerate}
Furthermore,  $(1)$ is a necessary condition. 

In Theorem~\ref{thm:what_tranched_are_quasi} we show that a tranched graph is a quasi-graph if and only if it is arcwise connected and has a finite depth (see Definitions~\ref{def:3.6}, \ref{def:3.7} and \ref{def:4.21}). The result is strengthened by Example~\ref{exmp:infinite_depth}, which shows that the set of assumptions is optimal, by the construction of an arcwise connected generalized sin(1/x)-type continuum with infinite depth.  As such, we get a characterization of generalized sin(1/x)-type continua that are also quasi-graphs, giving us a result opposite to Theorems~\ref{when_quasi_is_sin} and~\ref{thm:commonbase}. 

The paper is organized as follows. We recall the definitions and basic results in Section~\ref{sec:2}. The conditions for a quasi-graph to be a generalized sin(1/x)-type continuum are presented in Section~\ref{sec:3}, while Section~\ref{sec:4} contains conditions for the opposite inclusion. The entirety of Section~\ref{sec:example} is devoted to rigorous construction of an arcwise connected generalized sin(1/x)-type continuum which is a tranched graph of infinite depth. In Section~\ref{sec:dynamics} we present how the structure of the spaces considered in the paper impacts the dynamics on them.

\section{Preliminaries} \label{sec:2} 
    We denote by $\N=\{1,2,,\ldots\}$, $\R=(-\infty,\infty)$, and $\R_+=[0,\infty)$ the sets of natural numbers, real numbers, and non-negative real numbers, respectively.
    
 Let $X,Y$ be compact metric spaces.  We say that a continuous map $f \colon X \ra Y$ is \textit{monotone} if $f^{-1}(y)$ is a connected subset of $X$ for every point $y \in Y$. To distinguish between degenerate and nondegenerate sets $f^{-1}(y)$ induced by a continuous monotone map $f$ in $X$, we will use the term \textit{fiber} of $f$ for any set $f^{-1}(y)$, $y\in Y$ and reserve the word \textit{tranche} for fibers
that are not singletons.
     
    We write $\overline{U}$ for the closure of $U$, $\Inter U$ for the interior of $U$ and $\partial U=\overline{U} \backslash \Inter U$ for the boundary of $U$. 
    We say that a set is \textit{meager} if it can be written as a countable union of nowhere dense sets and we call it \textit{residual} if its complement is meager. 
    
     For a given metric space $(X,d)$ by $d_H$ we denote the \textit{Hausdorff metric} induced by $d$ 
     on the space $2^X$ of compact non-empty subsets of $X$ (see \cite{Hyperspaces} for more details). 
     
By the \textit{Hilbert cube} we mean the space $\mathcal{H}=[0,1]^\N$ equipped with the product metric $d(x,y)=\sum_{i=1}^\infty 2^{-i}|x_i-y_i|$.  By a \textit{continuum} we mean any compact, connected metrizable space. We assume that the reader is familiar with continuum theory (e.g., see \cite{Nadler} as a standard reference), and we only very briefly present some of the concepts to make the paper more self-contained.
    The hyperspace of all subcontinua of a continuum $X$
    is denoted $\mathcal{C}(X)\subset 2^X$. It is well known that it is a closed subset of $2^X$ and therefore is also compact \cite{Hyperspaces}.

    An \textit{arc} is any continuum homeomorphic to the interval $[0,1]$ with natural topology and  a \textit{(topological) graph}  is the union of a finite collection of arcs (called \textit{edges}), intersecting only at their endpoints (called \textit{vertices}). Note that, by the definition, the vertices of any edge are distinct points.    A topological graph that does not contain any \textit{circle} (a subset homeomorphic to the unit circle with the natural topology induced from the plane) is called a \textit{tree}. Let $X$ be an arcwise connected space and let $x,y\in X$ be two distinct points. If there is a unique arc $J\subset X$ with endpoints $x$ and $y$ then we denote $[x,y]=J$. It is not difficult to verify that if $X$ is a tree, then $[x,y]$ is defined for two distinct $x,y\in X$.

     A \textit{Peano continuum} is any locally connected continuum (that is, every point has an arbitrarily small open and connected neighborhood). It is well known that being a locally connected continuum is equivalent to being an image of the unit interval $[0,1]$ under some continuous mapping.  For a more detailed exposition on Peano continua, the reader is referred to \cite{Nadler}.

    Denote by $S_n$ a topological graph obtained as the union of $n$ edges that share a common endpoint $s_n\in S_n$, $n\geq 2$.
   For consistency, denote $S_1=[0,1]$ and $s_1=0$. We say that $X$ is an $n$-star centered at $x$ if there is a homeomorphism $h \colon X \ra S_n$ with $h(x)=s_n$.
        \begin{defn}
            Let $X$ be a nondegenerate, arcwise connected 
             continuum and let $x\in X$, then:
(i) the \textit{valence} of $x$ is the number $$\text{val}(x)=\sup_{k \in \N} \{ k: \text{ there is a $k$-star contained in $X$ centered at $x$}  \};$$ 
(ii) an \textit{endpoint} is any point with valence equal to 1;
(iii) a \textit{branching point} is any point with valence greater than 2.

        We write $\End(X), \Br(X)$  to denote the set of endpoints and the set of branching points of $X$, respectively.
        \end{defn}
  
    The original definition of quasi-graphs in \cite{MR3557770} states that it is an arcwise connected continuum and there is a natural number $N$ such that for each arcwise connected subset $Y$, the set $\overline{Y} \backslash Y$ has at most $N$ arcwise connected components. For the purpose of the present paper, we will use the equivalent characterization provided by \cite[Theorem 2.24]{MR3557770}.  Let us first introduce the necessary terminology.

	\begin{defn}
		Let $X$ be a compact metric space. We say that a subset $L$ of $X$ is a \textit{quasi-arc} with a parametrization $\varphi$ if the map $\varphi \colon [0,\infty) \rightarrow L$ is a continuous bijection. We call the point $\varphi (0)$ 
   the \textit{endpoint} of $L$. We denote $\omega(L) = \bigcap_{m\geq 0}\overline{\varphi[m,\infty)}$, and say that a quasi-arc is \textit{oscillatory} if $\omega(L)$ has more than one element. It is easy to see that the endpoint and the limit set of a quasi-arc are independent of the parametrization. For the sake of consistency, unless stated otherwise, in all figures presented in the paper, we will always depict oscillatory quasi-arcs in blue color and their limit sets in red color.
	\end{defn}
    
        In the literature there is a well established notion of a  \textit{ray}, which is a space $L$ homeomorphic to $\R_+$. A continuum $X$ is called a compactification of a ray if it can be represented as union of a ray $R$ and continuum $P$ such that $R \cap P = \emptyset$ and $P=\overline{R}\setminus R$. The continuum  $P$ is called  \textit{remainder} of the compactification. One can easily verify that every ray is a quasi-arc, and that the class of quasi-arcs is strictly larger as it allows for self-accumulation (we allow $\omega(L) \cap L \neq \emptyset$). For example, both circle and Warsaw circle are quasi-arcs, but not rays. Notice that if we view a ray $L$ as a quasi-arc then, the limit set $\omega(L)$ is the remainder of $L$. For some results about ray compactifications see for example (\cite{MR3227203},\cite{MR1307490}) .
    
       Non-oscillatory quasi-arcs in $X$ are referred to as $0$-order quasi-arcs. We say that a quasi-arc $L$ is of \textit{ $k$-order} if it contains within its limit set $\omega(L)$ a $(k-1)$-order quasi-arc and the set $\omega(K)$ does not contain any $(k-1)$-order quasi-arcs for any quasi-arc $K \subset \omega(L)$. 
        \par
        If for any natural number $n \in \N$ there exists a sequence of oscillatory quasi-arcs:
        $$\{L_0,L_1, \ldots, L_n\}$$
        with $L_0=L$ and $L_{i+1} \subset \omega(L_i)$, then we say that $L$ is \textit{ $\infty$-order} oscillatory quasi-arc. Let $\varphi \colon [0,\infty) \to X$ be a parametrization of quasi-arc $L$. If for every $t \in \N$, the quasi-arc $\varphi([t,\infty))$  is not a subset of $\omega(K)$ for any oscillatory quasi-arc $K$, we will say that $L$ is \textit{without ancestors}.

        We are now prepared to give a formal definition of quasi-graphs (see Theorem 2.24 in \cite{MR3557770}).
	\begin{defn}
		\label{def_quasigraph}
		A \textit{quasi-graph} is a continuum $X$  that can be decomposed into a topological graph $G$ and pairwise disjoint oscillatory quasi-arcs $ L_1,..., L_n$ such that:
		\begin{enumerate}[(i)]
			\item\label{def_quasigraph:1} $X=G \cup \bigcup_{j=1}^{n} L_j$ and $\End(X) \cup \Br (X) \subset G$,
			\item\label{def_quasigraph:2} for each $ 0 \leq i \leq n$ $ L_i \cap G = \{a_i\}$, where $a_i$ is the endpoint of $L_i$,
			\item\label{def_quasigraph:3} $\omega (L_i) \subset G \cup  \bigcup_{j=1}^{i-1} L_j$ for each $ 0 \leq i \leq n$,
			\item\label{def_quasigraph:4} if $\omega (L_i) \cap L_j \neq \emptyset$ for some $ 0 \leq i,j \leq n$, then $\omega(L_i) \supset L_j$
		\end{enumerate}	
  We will denote $\omega(X)=\bigcup_{i=1}^n \omega(L_i)$. We can see that $\omega(X)$ is the union of nondegenerate nowhere dense subcontinua of $X$, hence it doesn't depend on the choice of quasi-arcs in the decomposition.
  	\end{defn}
         In other words, we can construct every quasi-graph in a finite number of steps as follows. We start with a topological graph and then in each step we add one by one a finite number of oscillatory quasi-arcs (without adding branching points at any step of the construction) such that their limit set is contained in the continuum generated in the previous step.
        
        Now let us define the second class of our interest: generalized sin(1/x)-type continua. The following definition can be found in Section 5 in \cite{MR3272777} and is a natural generalization of the notion of $\lambda$-continuum  studied already by Kuratowski (see \cite[\S 48, Ch. III, footnote on p.197]{MR259835}). In the context of $\lambda$-continua, the monotone map $\phi$ in the following definition is sometimes called a \textit{Kuratowski map} (see \cite{MR965302,MR1073777})
	\begin{defn}
		\label{def_sin1/x}
		A continuum $X$ is a \textit{generalized sin(1/x)-type continuum}  if there exists a topological graph $Y$ and a continuous monotone map $\phi \colon X \rightarrow Y$ with the following properties:
		\begin{enumerate}[(i)]
			\item\label{sin-i} $\phi^{-1} (y)$ is nowhere dense in $X$ for any $y\in Y$,
			\item\label{sin-ii} $\phi^{-1} (D)$ is dense in $X$, where $D = \{y \in Y$ such that $\phi^{-1}(y)$ is degenerate$\}$,
			\item\label{sin-iii} if $Y_0$ is a subcontinuum of  $\phi^{-1} (y)$ and $\epsilon >0$ then there exists an arc $[a,b] \subset Y$ such that $d_H(Y_0,\phi^{-1}([a,b]))<\epsilon$. 
			
		\end{enumerate}
		The sets $\phi^{-1} (y)$, $y\in Y$ are called \textit{fibers} (of $f$) in $X$, and fibers that are nondegenerate are called \textit{tranches} of $X$. We will often refer to 
  \eqref{sin-iii}
  as the
  \textit{approximation property}. As we show later in the paper, the decomposition into fibers doesn't depend on the choice of the graph $Y$ and map $\phi$ satisfying the conditions. 
	\end{defn}

The following is a simple, yet important, observation.

    \begin{lem}
    \label{lem:if_dense_then_fibers_ndense}
         Let $X,Y$ be nondegenerate continua and let $\phi \colon X \rightarrow Y$ be a continuous monotone map. If the set $\phi^{-1} (D)$ is dense in $X$, where $D = \{y \in Y$ such that $\phi^{-1}(y)$ is degenerate in $X\}$, then $\phi^{-1}(y)$ is nowhere dense in $X$ for every $y\in Y$.
    \end{lem}
    \begin{proof}
        The map $\phi$ is continuous, and hence all fibers $\phi^{-1}(y)$ are closed. Therefore, it is sufficient to show that the fibers have empty interiors. Assume on the contrary that there is a point $y_0 \in Y$ such that $\phi^{-1}(y_0)$ contains an open subset $U \subset \phi^{-1}(y_0)$. Since $X$ is a nondegenerate continuum, it follows that $\phi^{-1}(y_0)$ is a tranche of $X$. As the set of 
        $\phi^{-1}(D)$
        is dense in $X$, we can find a point $y_1$ with a degenerate preimage under $\phi$ such that $\phi^{-1}(y_1) \subset U$. This implies that $\phi^{-1}(y_1) \subset U \subset \phi^{-1}(y_0)$, leading to a contradiction. 
    \end{proof}
   \begin{rem}
       By Lemma~\ref{lem:if_dense_then_fibers_ndense} we get that $(i)$ in Definition~\ref{def_sin1/x} is redundant, so we will omit it in the rest of the paper.
   \end{rem}
    \begin{lem}\label{lem:nowheredense}
       Suppose a continuum $X$, topological graph $Y$, and a map $\phi \colon X \to Y$ satisfy $(ii)$ from Definition~\ref{def_sin1/x}. Then any fiber $\phi^{-1}(y)$ of the map $\phi$ is not a strict subset of any nowhere dense subcontinuum of $X$. 
   \end{lem}
   \begin{proof}
       Let $Y$ be a topological graph, and $\phi \colon X \ra Y$ be a continuous monotone map satisfying $(ii)$ from Definition~\ref{def_sin1/x}. Suppose that there is a fiber $\phi^{-1}(y)$ for which there is a nowhere dense subcontinuum $U \supset \phi^{-1}(y)$. Denote $V=\phi(U)$. Notice that we can take $U$ such that $V$ is a star with center in $y$. Choose now $z \in V \backslash End(V)$ that has a degenerate preimage. As $U$ is nowhere dense, there is a sequence $\{z_n\}_{n=1}^\infty \subset X \backslash U$ of elements of singleton fibers converging to $\phi^{-1}(z)$. Clearly each $\phi(z_n)\in X \backslash V$, and by continuity  the sequence $\{\phi(z_n)\}_{n=1}^\infty$ converges to $z$. By definition, $z$ is not an endpoint of the star $V$, which is a contradiction.
       \end{proof}
    From Lemma~\ref{lem:nowheredense} we immediately get the following:
    \begin{rem}\label{rem:collapse}
        Suppose $\phi_1 \colon X \to Y_1$ and $\phi_2 \colon X \to Y_2$  satisfy $(ii)$ from Definition~\ref{def_sin1/x}. Then for any $\Lambda $, a maximal (in the sense of inclusion) nowhere dense subcontinuum of $X$, there are $y_1 \in Y_1, y_2 \in Y_2$ such that $\phi_i^{-1}(y_i)=\Lambda$ for $i=1,2$
    \end{rem}

\section{Characterization of quasi-graphs that are generalized sin(1/x)-type continua}\label{sec:3}

In this section we will try to describe which quasi-graphs are generalized sin(1/x)-type continua. The simplest case of the Warsaw circle suggests
that the limit sets of quasi-arcs are good candidates for tranches, and this intuition may lead to a claim that it is a general property
(for example, see comments after Question 1.1 in \cite{MR4385436}). Unfortunately, this observation does not generalize onto all quasi-graphs as we will show later in this paper.  Another property that this simple example may suggest is that the limit sets of quasi-arcs satisfy the approximation property (i.e. \eqref{sin-iii} from Definition~\ref{def_sin1/x} holds). This turns out to be false in the general case as well.   To provide a simple example, we will refer to the so-called class $\ClassW$.

\begin{rem}
\label{rem:arclike_are_classw}
Over the years, many continua belonging to \ClassW{} have been discovered, including arc-like continua, non-planar circle-like continua, hereditarily indecomposable continua, and many others (see \cite[Ch. IX, Sec. 67]{Hyperspaces}).
\end{rem}

\begin{defn}\label{def:classW}
    A continuum is said to be in $\ClassW{}$, written  $X\in \ClassW{}$, if for every continuum $S$ and any surjective continuous map $f \colon S \to X$, any subcontinuum of $X$ is the image of a subcontinuum of $S$.
\end{defn}

The following is one of the  conditions equivalent to  Definition~\ref{def:classW} (see \cite{GT,Proctor}, cf. \cite[Ch. VIII, Sec. 35]{Hyperspaces}). 

\begin{lem}
    A continuum $X$ is an element of  $\ClassW{}$ if and only if every compactification $Y$ of $[0,1)$ with the remainder $X$ has the property that $\mathcal{C}(Y)$ is a compactification of $\mathcal{C}([0, 1))$.
\end{lem}

 It immediately follows from the above that if a continuum is not in the \ClassW{} then there is an arc $[0,1)$ and its compactification without the approximation property (hence it is not a generalized sin(1/x)-type continuum). It is well known that $n$-stars with $n\geq 3$ and circles do not belong to \ClassW{};  to see this, consider as an example continua from Figures~\ref{fig:branching_point_no_sin} and \ref{fig:limiting_circle}. 
  Any monotone map from the definition of sin(1/x)-type continuum has to collapse red subcontinua to a point by Remark~\ref{rem:collapse}.
 As a consequence we see that both examples are not generalized sin (1/x)-type continua, while the continuum on Figure~\ref{fig:limiting_circle} is a quasi-graph.

\begin{figure}
\centering
\begin{minipage}{.5\textwidth}
             \centering
\includegraphics[scale=0.35]{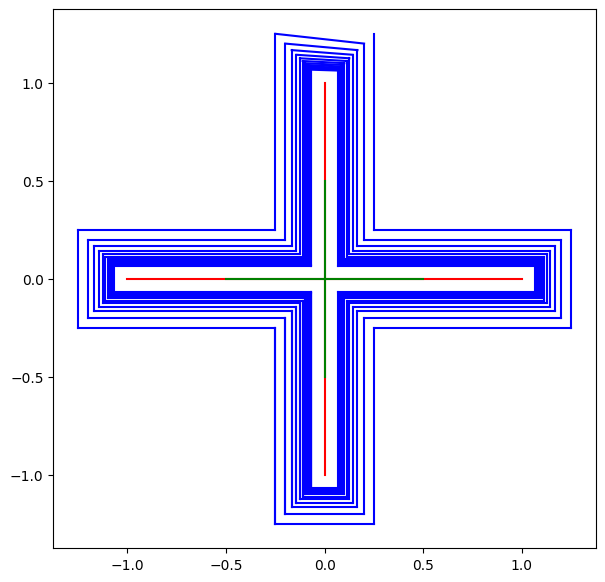}
            \caption{A quasi-arc with $4$-star as the limit set. If a map $\phi$ collapses the $4$-star to a point, then approximation property is violated, e.g. by subcontinuum marked in green.}

            \label{fig:branching_point_no_sin}
\end{minipage}%
\begin{minipage}{.5\textwidth}
   \centering
            \includegraphics[scale=0.25]{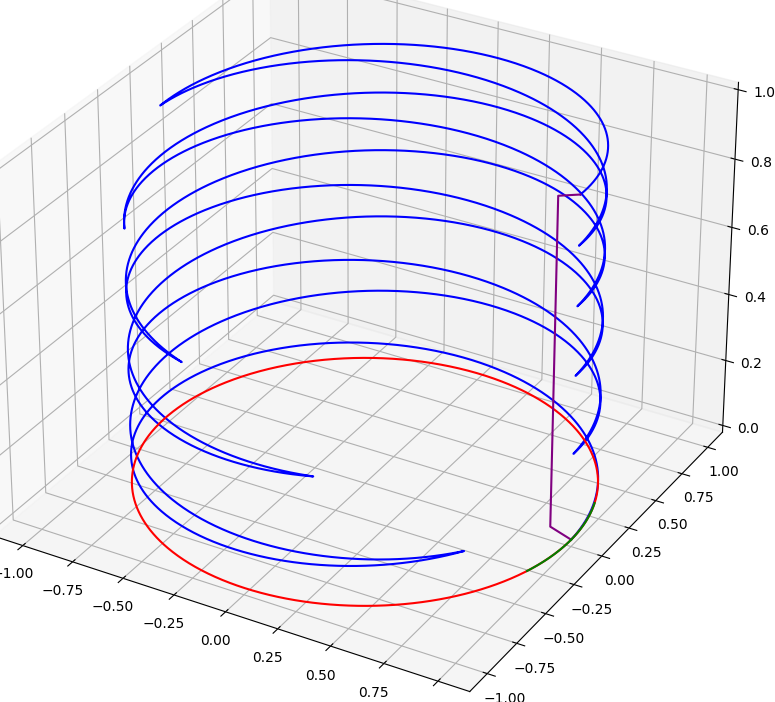}
            \caption{A quasi-graph whose limit set is circle, but is not a generalized sin(1/x)-type continuum.}            \label{fig:limiting_circle}
\end{minipage}
\end{figure}

\begin{lem}
        \label{limit_collapse}
		Suppose that $X$ is a quasi-graph, $Y$ is a topological graph, and $\phi \colon X \rightarrow Y$ is continuous and monotone. 
  Suppose that $\Lambda$ is a connected component of the set $\omega(X)$.
  Then $\phi(\Lambda) = \{y\}$ for some $y \in Y$. 
	\end{lem}
    \begin{proof}   
		Suppose on the contrary that there is a quasi-graph $X$, a graph $Y$ and a  continuous monotone  map between them $\phi \colon X \rightarrow Y$, and assume that there is a connected component $\Lambda$ of 
   $\omega(X)=\bigcup_{j=1}^{n} \omega(L_i)$
  with $a,b \in \Lambda$, $\phi(a) \neq \phi(b)$.  
  First, assume that $a$ and $b$ are elements of the limit set of the same quasi-arc $L_k$ and fix its parametrization $\varphi \colon [0,+\infty) \rightarrow L_k$.  
    Let $ \{a_n \}_{n=1}^\infty, \{b_n\}_{n=1}^\infty \subset L_k$ be the sequences that converge to $a$ and $b$, respectively. There is a sequence $s_1<t_1<s_2<t_2< \ldots < s_n<t_n< \ldots$ such that $\{\varphi([s_n,t_n])\}_{n=1}^\infty$ are disjoint arcs, while $\varphi(s_n)=a_n$ and $\varphi(t_n)=b_n$. Then also $J_n = \phi(\varphi([s_n,t_n]))$ is an arc and since $\phi(a) \neq \phi(b)$, there is $\delta>0$ such that
    $$\diam(J_n) \geq d(\phi(a_n),\phi(b_n)) \geq d(\phi(a),\phi(b))/2 > \delta$$
    for all sufficiently large $n$.
     Arcs $J_n$ cannot be pairwise disjoint for different $n$, and so there is $x \in \Inter J_n \cap  \Inter J_m$ for some arbitrarily large $m\neq n$. This means that $\phi^{-1}(x) \cap \varphi([s_n,t_n]) \neq \emptyset$ and $\phi^{-1}(x) \cap \varphi([s_m,t_m]) \neq \emptyset$ while $a_n,a_m \not\in\phi^{-1}(x)$ or $b_n,b_m \not\in\phi^{-1}(x)$. This immediately implies that $\phi^{-1}(x)$ is not connected.  This is a contradiction.

 We obtained that there is no quasi-arc such that both $a$ and $b$ belong to its limit set. Let 
  $\Lambda = \omega(L_a) \cup \omega(L_b) \cup \omega(L_{a_1}) \cup ... \cup \omega(L_{a_m})$, where $\omega(L_a),\omega(L_b), \omega(L_{a_i})$
are limit sets of distinct quasi-arcs,
 with $a \in \omega(L_a)$ and $b \in \omega (L_b)$. We already proved that $\phi(\omega(L_a))=\{y\}$ for some $y\in Y$. Suppose now that $\omega(L_a) \cap \omega(L_\xi) \neq \emptyset$ for some $\xi \in \{ a_1, a_2, ...,a_m, b\}$. This means that for some $x_k \in \omega(L_\xi)$ we have $\phi(x_k)=y$, which means $\phi(\omega(L_\xi))= \{y\}$. Then either $\xi=b$, in which case we have a contradiction, or we repeat the above argument until we reach $\omega(L_b)$ (showing $\phi(\omega(L_a))= \{y\}=\phi(\omega(L_b))$), which must eventually occur as $\Lambda$ is connected.
    \end{proof}

The following is a standard but useful tool for the construction of factors (see
Theorem~6 from $\S 19$ in \cite{MR217751}). We will use it for projecting quasi-graphs onto topological graphs.

    \begin{thm}
        Let $D$ be a decomposition of a Hausdorff space $X$ and let $\rho$ be the corresponding equivalence relation. If  $D$ is  lower semi-continuous and $\rho$ is closed, then $X/ _{\rho}$ equipped with the quotient topology is a Hausdorff space.
    \end{thm}
 With the above tool at hand, we easily obtain the following result. Details of its (standard) proof are left to the reader.
        \begin{lem}
        \label{lem:graph_for_quasi}
             Let $X$ be a quasi-graph. Define an equivalence relation $\sim$ in $X$ by $a \sim b$ if $a=b$ or there exists a connected component $\Lambda$ of $\omega(X)$
            such that $a,b \in \Lambda$. Then $X/_\sim$ is a topological graph.
     
        \end{lem}

    The continuum depicted in Figure~\ref{fig:limiting_circle}  shows that the approximation property of generalized sin(1/x)-type continua in Definition~\ref{def_sin1/x}\eqref{sin-iii} sometimes fails for quasi-graphs. Close investigation of these examples shows, however, that the set of singleton fibers is always dense in the quasi-graph,  so the condition in Definition~\ref{def_sin1/x}\eqref{sin-ii} is satisfied. This motivates the following definition.
\begin{defn}\label{def:3.6}
		\label{def:tranched_graph}
        A continuum  $X$ is said to be a  \textit{tranched graph}  if there is a topological graph $Y$ and a continuous monotone map $\phi \colon X \rightarrow Y$  such that $\phi^{-1} (D)$ is dense in $X$, where $D = \{y \in Y$ such that $\phi^{-1}(y)$ is degenerate$\}$.
\end{defn}

Next we will show that in a tranched graph, set of singleton fibers is not only dense, but residual, meaning tranched graphs are similar to graphs on a topologically large set.
\begin{thm}
\label{thm:set_of_tranches_is_meager}
 Let $X$ be a tranched graph and let $\phi \colon X \ra Y$ be an associated mapping. 
The set
$\phi^{-1} (D)$ is residual in $X$, where $D = \{y \in Y$ such that $\phi^{-1}(y)$ is degenerate$\}$.
\end{thm}
\begin{proof}
    Let $X$ be a tranched graph and let $\phi \colon X \ra Y$ be an associated mapping. 
    Let $x\in X$ be a point in a singleton fiber, i.e. $\phi^{-1}(\phi(x))=\{x\}$ and fix any $\eps>0$. Observe that there is $\delta>0$ such that if $z\in B(x,\delta)$ then $\diam \phi^{-1}(\phi(z))<\eps$. Otherwise, there are points $w_n,z_n$ such that $\phi(z_n)=\phi(w_n) \to \phi(x)$ with $d(z_n,w_n)\geq \eps/2$ and as a consequence there are $w\neq z$ with $\phi(x)=\phi(w)=\phi(z)$, which is a contradiction. This shows that the set $$ \mathcal{D}_\epsilon = \{x\in X : \diam \phi^{-1}(\phi(x))<\eps \}$$
    is open and dense.

If we denote by $\mathcal{D}_0$ the set of points belonging to degenerate fibers (i.e. $\mathcal{D}_0=\phi^{-1} (D)$), then obviously:
    $$\mathcal{D}_0= \bigcap_{\epsilon \in \Q_+} \mathcal{D}_\epsilon.$$

Indeed, the set $\mathcal{D}_0$ is residual, which completes the proof.
\end{proof}
Now we can prove that properties of a tranched graph don't depend on the choice of map onto a topological graph, as we can go from one to the other by a homeomorphism.
\begin{thm}
\label{thm:two_maps_are_homeo}
Assume that $X$ is a tranched graph, let $Y_1,Y_2$ be two topological graphs and let $\phi_i\colon X
\to Y_i$, $i=1,2$ be two (possibly different) continuous monotone maps from the definition of tranched graph. Then there is a homeomorphism $\psi\colon Y_1\to Y_2$ such that $\psi\circ \phi_1=\phi_2$.
\end{thm}
\begin{proof}
By Theorem~\ref{thm:set_of_tranches_is_meager}, sets $D_1$ and $D_2$ of elements of degenerate fibers of  $\phi_1$ and $\phi_2$ respectively are both residual, hence their intersection is a residual set $D = D_1\cap D_2 \subset X$, in particular $\phi_1^{-1}(\phi_1(x))=\phi_2^{-1}(\phi_2(x))=\{x\}$ for every $x\in D$.

Let $R$ be the closure of the relation $\{(\phi_1(x),\phi_2(x)): x\in D\}$ in $Y_1\times Y_2$. It is clear that $Y_1,Y_2$ are projections of $R$ onto respective coordinates. We claim that $R$ is one-to-one. To see this, fix any $(a,b)\in R$ and let $A=\phi_1^{-1}(a)$, $B=\phi_2^{-1}(b)$. There is a sequence $\{x_n\}_{n=1}^\infty\subset D$  with $x=\lim_{n \to \infty}x_n $  such that $(\phi_1(x_n),\phi_2(x_n))\to (a,b)$, hence $x \in A\cap B\neq \emptyset$. Then $b\in \phi_2(A)$ and $\phi_2(A)$ must be degenerate as otherwise $D\cap A\neq \emptyset$ and $A$ are nondegenerate which is impossible. This shows that $A\subset B$ and by symmetric argument $B\subset A$. Indeed $R$ is one-to-one, and hence induces a bijection $\psi\colon Y_1\to Y_2$.

To see that $\psi$ is continuous, fix any sequence $\{ \hat{x}_n\}_{n=1}^\infty$ such that $\hat{x}_n\to \hat{x}$ in $Y_1$. By previous argument, $\phi_2(\phi_1^{-1}(\hat{x}_n))$ is a single point $\hat{y}_n=\psi(\hat{x}_n)$. Let $\hat{y}$ be any limit point of the sequence $\{\hat{y}_n\}_{n=1}^\infty$. For every $n$ there is a point $x_n\in D$ such that 
$$
\dist(x_n,\phi_1^{-1}(\hat{x}_n))=\dist(x_n,\phi_2^{-1}(\hat{y}_n))<1/n
$$
and $x_n\to x\in \phi_2(\hat{y})$. But by continuity and uniqueness of the limit we have that $\psi_1(x_n)\to \hat{x}$.
Since $(\phi_1(x_n),\phi_2(x_n))\in R$, we have $(\hat{x},\hat{y})\in R$ completing the proof.
\end{proof}
As such, we get the following:
\begin{rem}
    By Theorem~\ref{thm:two_maps_are_homeo} decomposition of a tranched graph into fibers, up to homeomorphism, does not depend on the choice of topological graph $Y$ and map $\phi \colon X \to Y$. Therefore we can speak about fibers and tranches of $X$ without mentioning $\phi$, since it does not lead to ambiguity.
\end{rem}

Recall that our first goal is to characterize what generalized sin(1/x)-type continua are quasi-graphs. The examples presented indicate that we should put some restrictions on the topological structure of the tranches. 
Definition~\ref{def:tranched_graph} for example, allows a square to be a tranche, which we have to eliminate as it cannot be a tranche for any quasi-graph. In Theorem~\ref{thm:sin1x_with_any_quasig_as_tranche} we will show that the definition of a generalized sin(1/x)-type continua does not eliminate this as well. To deal with this problem, we introduce the following hereditary property that imposes some restrictions on tranched graphs.

\begin{defn}\label{def:3.7}

We say that a continuum $X$ is a \textit{tranched graph with hereditary fibers} if $X$ is a tranched graph with an associated continuous monotone map $\phi \colon X \rightarrow Y$ onto a topological graph $Y$, then every fiber $\phi^{-1}(y)$, $y\in Y$ is a singleton or a tranched graph.
    A continuum $X$ is \textit{hereditary tranched graph} if any subcontinuum $A \subset X$ is either a singleton or a tranched graph with hereditary fibers.
\end{defn}

It is obvious that all generalized sin(1/x)-type continua are tranched graphs, and the next lemma shows that quasi-graphs also fall into this class.
The definition of tranched graph is general enough to cover all quasi-graphs and all generalized sin(1/x)-type continua.
Our further goal is to impose additional conditions on tranched graphs to characterize the intersection of the classes of quasi-graphs and generalized sin(1/x)-type continua.

\begin{lem}
\label{lem:quasi-graphs_are_tranched}
 Let $X$ be a quasi-graph.
Define the equivalence relation $\sim$ on $X$ by $a \sim b$ if $a=b$ or there exists a connected component $\Lambda$ of $\omega(X)$
such that $a,b \in \Lambda$.
Then $X/_\sim$ is a topological graph and $X$ is a hereditary tranched graph.  Furthermore, the quotient map $\phi \colon X \ra X /_\sim$ satisfies Definition \ref{def:tranched_graph} for $X$.
\end{lem}
  \begin{proof}
 Denote $Y=X / _\sim$ and observe that by
Lemma~\ref{lem:graph_for_quasi} we obtain that $Y$ is a topological graph. Let $\phi \colon X \ra Y$ be the quotient map induced by $\sim$. Tranches in this case are nondegenerate equivalence classes of $\sim$ which are connected components of $\omega(X)$. By Definition~\ref{def_quasigraph} there is a finite number of them. This means that the set $\phi^{-1} (D)$, where $D = \{y \in Y$ such that $\phi^{-1}(y)  \text{ is degenerate} \}$  is dense in $X$. 
        \par
 Recall the original definition of quasi-graphs $X$ from \cite{MR3557770} which states that it is an arcwise connected continuum and there is a natural number $N$ such that for each arcwise connected subset $Y$, the set $\overline{Y} \backslash Y$ has at most $N$ arcwise connected components. It is clear that this property is hereditary by any arcwise connected subcontinua of $X$, hence any
 arcwise connected subcontinuum of a quasi-graph is a quasi-graph, which ensures the hereditary property for arcwise connected tranches.
 
 Fix a decomposition $X=G \cup \bigcup_{i=1}^N L_i$ and any tranche $T=\phi^{-1}(y)$.
As  $T\subset \omega(X)$ and is connected by the definition, it is the union of finitely many arcwise connected components, $T=G_0 \cup L_{\alpha_1} \cup \ldots \cup L_{\alpha_n}$, $\alpha_i \in \{1, \ldots, N\}$, $G_0 \subset G$.
Namely by Definition~\ref{def_quasigraph}\eqref{def_quasigraph:4} for every $i$, either $T\cap L_i=\emptyset$
or $L_i\subset T$.
Since endpoints of quasi-arcs are elements of $G$, we see that  $S=G \cup L_{\alpha_1} \cup \ldots \cup L_{\alpha_n}$ is arcwise connected, hence a tranched graph by previous argument.
Then it is easy to see that $T$ in that case is also a tranched graph as $S \setminus T$ is a finite union of topological graphs. We can repeat the above construction (with the analogue of the relation $\sim$) for every tranche $T$ of $X$. 
Suppose now that $X_0$ is a subcontinuum of $X$. If $X_0$ is not a subset of a tranche, then the map $\phi$ restricted to $X_0$ satisfies the conditions  from Definition \ref{def:tranched_graph}. Suppose now that there is a tranche $T_1$ in $X$ such that $X_0 \subset T_1$. We know that $T_1$ is a tranched graph, furthermore by adding finite number of arcs we get a quasi-graph with the same set of tranches. It follows that $T_1$ has hereditary fibers. If $X_0$ is not nowhere dense in $T_1$, we get the result as before. If not, it is a subset of a tranche $T_2$ of $T_1$, which by our argumentation is also a tranched graph with hereditary fibers. We continue this process until we find $T_m$ such that $X_0$ is not nowhere dense in $T_m$.

By the definition of quasi-graph such $T_m$ must exist.
    \end{proof}

We already proved that all quasi-graphs are tranched graphs. We only need to investigate the topological structure of limit sets of oscillatory quasi-arcs in the context of Definition~\ref{def_sin1/x}\eqref{sin-iii}. 
\begin{thm}
\label{when_quasi_is_sin}
		Let  $X=G \cup \bigcup_{j=1}^{n} L_j$ be a quasi-graph. Assume that for every connected component $\Lambda$ of $\omega(X)$
  the following assertions hold:
        \begin{enumerate}
            \item\label{when_quasi_is_sin:1} There is a quasi-arc $L$ in $X$ such that $\omega(L)=\Lambda$ and
            \item  Continuum $\Lambda$  belongs to \ClassW{}.
        \end{enumerate}
        Then $X$ is a generalized sin(1/x)-type continuum.
        
	\end{thm}
	\begin{proof}
         Let $Y= X/_\sim$, where the relation $\sim$ is defined on $X$ as in Lemma~\ref{lem:graph_for_quasi}, that is, put $a \sim b$ if $a=b$ or there exists a connected component $\Lambda$ of 
             $\omega(X)$
            such that $a,b \in \Lambda$.
         Let $\phi \colon X \rightarrow Y$ be the associated quotient map. By Lemma~\ref{lem:quasi-graphs_are_tranched} we see that $X$ is a tranched graph, so it remains to show that the conditions from Definition~\ref{def_sin1/x}\eqref{sin-iii} hold.

         Fix any $y \in Y$. If $\phi^{-1}(y)=\{x\}$, there exist neighborhoods $U,V$ of $x,y$ respectively such that $\phi|_U \colon U \rightarrow V$ is a homeomorphism. Now suppose that $\phi^{-1}(y)$ is nondegenerate. By assumptions $\phi^{-1}(y)= \Lambda=\omega(L_c)$ for some quasi-arc $L_c$. 
         This means that $\Lambda\in \ClassW{}$ is a compactification of $L_c$. It follows that for any subcontinuum $Y_0 \subset \Lambda$ there is a sequence of arcs $\{[a_n,b_n]\}_{n=1}^\infty$ in $L_c$ such that $\{[a_n,b_n]\}_{n=1}^\infty$ converges to $Y_0$ in the Hausdorff metric. One can easily see that this is  implies the approximation property.
	\end{proof}

 Remark~\ref{rem:arclike_are_classw} recalls most known examples of elements of $\text{Class}(W)$. The following may provide another tool for identifying continua in this class.

\begin{rem}
\label{rem:classw_is_hereditary}
The result of Grispolakis and Tymchatyn \cite[Corollary~3.4]{GT} shows that if $Y=L\cup \omega(L)$ for some quasi-arc $L$ then $Y\in \text{Class}(W)$ iff $\omega(L)\in \text{Class}(W)$. It allows hierarchical constructions of generalized sin(1/x)-type continua.
\end{rem}

 Unfortunately, Theorem~\ref{when_quasi_is_sin} provides only a sufficient condition while a complete characterization seems almost an impossible task from the point of view of
Theorem~\ref{thm:sin1x_with_any_quasig_as_tranche} presented later in the paper. Roughly speaking, it says that any
continuum can be a tranche of a generalized sin(1/x)-type continuum. 
For example, we can construct a quasi-graph that has a $4$-star as a limit set of its unique oscillatory quasi-arc and satisfies the definition of generalized sin(1/x)-type continuum, e.g. see Figure~\ref{fig:limiting_star}. The only difference compared to Figure~\ref{fig:branching_point_no_sin}
is how the quasi-arc approaches its limit set.
     	\begin{figure}[ht]

            \centering
            \includegraphics[scale=0.38]{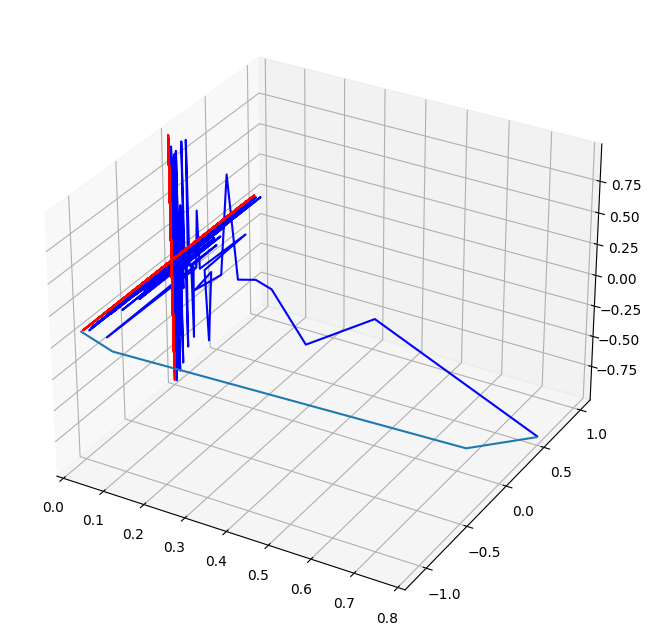}
               \caption{A quasi-graph which is generalized sin(1/x)-type continuum and contains 4-star as a tranche}         \label{fig:limiting_star}
        \end{figure}
 Then belonging to $\ClassW{}$ in Theorem~\ref{when_quasi_is_sin} is only
a sufficient condition ensuring that the resulting space is sin(1/x)-type continuum.On the other hand, the next result shows that condition \eqref{when_quasi_is_sin:1} in Theorem~\ref{when_quasi_is_sin} is necessary.

\begin{lem}
\label{lem:finite_tranches}
    Let $X$ be a tranched graph with finitely many tranches. Then every tranche is a union of limit sets of finitely many oscillatory quasi-arcs. Moreover, if $X$ is a generalized sin(1/x)-type continuum, then every tranche is the limit set of some oscillatory quasi-arc.
\end{lem}
\begin{proof}
     Let $X$  be a tranched graph with finitely many tranches and let $\phi \colon X \rightarrow Y$  be the associated map from Definition~\ref{def:tranched_graph}. 
     Fix any tranche $T\subset X$. If we denote $y=\phi(T)$, then there 
    is an open set $U\ni y$ such that $\overline U$ contains at most one branching point. 
    As there are finitely many tranches, we can pick $\epsilon>0$ such that $B=B(T,\epsilon)$ does not intersect any tranches other than $T$ and
    $\phi(B)\subset U$.       
    
      Now let $n=val(y)$ and choose a closed set $V\subset Y$, so that $U\supset V=E_1 \cup E_2 \cup \ldots \cup E_n$ is an $n$-star with the center $y$ and edges $E_i$. By continuity of $\phi$ and the fact that it is one-to-one on $B\setminus T$,  the sets $T_i=\phi^{-1}(E_i\setminus \{y\})$ are quasi-arcs for $i=1, \ldots n$ and $\bigcup_{i=1}^n \omega(T_i) \subset T$. Assume now there is $x \in T$ that is not an element of the limit set of a quasi-arc. Then, by the definition of topological limit, there is a ball centered at $x$ that does not intersect any quasi-arc. Choosing the radius of this ball to be smaller than $\epsilon$ if necessary, we get an open ball contained in $T$, contradicting the assumption that fibers are nowhere dense in $X$.
      Summing up, we get that  $T=\bigcup_{i=1}^n \omega(T_i)$, which proves the first part of the statement of theorem.
      \par 

      Suppose now that $X$ is a generalized sin(1/x)-type continuum  but there is a tranche $\Lambda \subset X$ that is not the limit set of any quasi-arc, i.e $\Lambda \neq \omega(L_k)$ for any oscillatory quasi-arc $L_k \subset X$. By previous argument we know that $\Lambda=\bigcup_{i \in K} \omega(L_i)$, where $K$ is the set of indices of quasi-arcs whose limit sets are subsets of $\Lambda$. 
      Each $\omega(L_k)$ is connected, thus $K$ is well defined.
      Denote $\delta = \min\limits_{k\in K} \diam(\Lambda \setminus \omega(L_k))>0$. 
            
    First we claim that if $Y_0  \subsetneq\Lambda$ is a  proper subcontinuum of $\Lambda$ then for $\epsilon>0$ small enough, any arc $[a,b]$ such that $d_H(Y_0,\phi^{-1}([a,b]))<\epsilon$ may intersect only one quasi-arc. To see this, suppose  that $[a,b] \subset Y$ is an arc with desired property such that $\phi^{-1}(a)$ and $\phi^{-1}(b)$ are singletons and elements of different quasi-arcs. Therefore, since $\epsilon$ is small, we must have that  $\Lambda \cap \phi^{-1}([a,b])\neq \emptyset$ and consequently $\Lambda \subset \phi^{-1}([a,b])$. This implies that:
     $$\epsilon>d_H(\phi^{-1}([a,b]),Y_0) \geq d_H(\Lambda,Y_0)>0$$
    But $\epsilon$ can be arbitrarily small, hence we may assume that $\epsilon<d_H(Y_0,\Lambda)$ which is a contradiction. Indeed, the claim holds.

     Now choose a subcontinuum $Y_0  \subset \Lambda $ such that $0<\diam(\Lambda \setminus Y_0 )< \delta$ and fix small $\epsilon>0$  provided by the claim above and assume that $\epsilon<\delta$. Since $X$ is a generalized sin(1/x)-type continuum, there is an arc $[a,b] \subset Y$ such that $d_H(\phi^{-1}([a,b]),Y_0)<\epsilon$.
      By the choice of $\epsilon$ we cannot have $Y_0\subset \phi^{-1}([a,b])$, thus 
     we can assume that $[a,b] \subset L$ for some quasi-arc $L$ with $\omega(L) \subset \Lambda$.  
     Passing with $\eps$ to $0$ and using the pigeon-hole principle, we obtain that $Y_0 \subset \omega(L)$ for some quasi-arc $L$    and therefore 
     $\diam(\Lambda \setminus Y_0) \geq \diam ( \Lambda \setminus \omega(L))$.  It follows that:
     $$\delta  > \diam(\Lambda \setminus Y_0) \geq \diam ( \Lambda \setminus \omega(L)) \geq \delta$$  which is a contradiction. The proof is finished.    
\end{proof}

        \begin{thm}\label{thm:commonbase}
            Let $X$ be a quasi-graph that is a generalized sin(1/x)-type continuum. Then for every connected component $\Lambda$ of  $\omega (X)$
            there is a quasi-arc $L \subset X$ such that $\omega(L)=\Lambda$.
        \end{thm}
        \begin{proof}
            By Lemma~\ref{lem:quasi-graphs_are_tranched} every quasi-graph is a hereditary tranched graph with tranches being connected components of  $\omega(X)$, in particular with finitely many tranches.
Lemma~\ref{lem:finite_tranches} 
completes the proof.
        \end{proof}

In the introduction, we were trying to convince the reader that the main goal behind the definition of a generalized sin(1/x)-type continuum was to obtain a nice generalization of topological graphs. In a sense $\phi \colon X\to Y$ should maintain structure of $X$ ,,similar'' to $Y$. Surprisingly, there is no direct restriction on the structure of a single tranche in a generalized sin(1/x)-type continuum. It can be any continuum, extending the class of generalized sin(1/x)-type continua much beyond the initial intuition.
The following result is folklore and is not difficult to prove. Details are left to the reader (cf. the arguments in the proofs in \cite{Proctor}).

   \begin{thm}
   \label{thm:sin1x_with_any_quasig_as_tranche}
        Suppose $X$ is a continuum. Then there exists an oscillatory quasi-arc $L$ with $\omega(L)$  homeomorphic to $X$ and such that  $Z=     \omega(L)\cup L$ is a generalized sin(1/x)-type continuum.
    \end{thm}

    \begin{rem}
        Suppose $X=G \cup \bigcup_{j=1}^{n} L_j$ is a quasi-graph. 
        Then there is an oscillatory quasi-arc $L_{n+1}$ such that  $\omega(L_{n+1})$ is homeomorphic to $G \cup \bigcup_{j=1}^{n} L_j$ and $Z=L_{n+1}\cup \omega(L_{n+1})$ is a quasi-graph and a generalized sin(1/x)-type continuum.
        Furthermore, if $X\subset \mathcal{H}\times \{0\}\times \{0\}$ then $L_{n+1}$ can be constructed in such a way that $Z=G \cup \bigcup_{j=1}^{n+1} L_j$  (i.e. we have exact representation without need of passing through a homeomorphism. Namely, we can use two first coordinates to define a spiral compactified by $\{0\}\times \{0\}$ and use its consecutive segments as parametrizations of arcs approximating $X$ in $\mathcal{H}$. This way we obtain an oscillatory quasi-arc with reminder $X$.
    \end{rem}
    
 So far, we provided necessary conditions ensuring that a given quasi-graph is a generalized sin(1/x)-type continuum (see Theorem~\ref{when_quasi_is_sin}) and that one of the conditions is necessary (see Theorem~\ref{thm:commonbase}). Lastly, we proved that any continuum can be a tranche  (see Theorem~\ref{thm:sin1x_with_any_quasig_as_tranche}).  Taking these results all together, it seems that the  characterization in full generality when a quasi-graph is also a generalized sin(1/x)-type continuum is out of reach.
    
        \section{Characterization of generalized sin(1/x)-type continua that are quasi-graphs }\label{sec:4}

Recall that quasi-graphs which are generalized sin(1/x)-type continua must have a finite number of tranches, because their tranches are connected components of limit sets of oscillatory quasi-arcs, and quasi-graphs have a finite number of those. It may happen, however, that a generalized sin(1/x)-type continuum has even uncountably many tranches (see \cite[Example 26]{MR3272777}).  As a complement to these situations, we will construct a generalized sin(1/x)-type continuum with an infinite ,,depth'' later in this section (see Lemma~\ref{lem:infite_depth}).

\begin{lem}
\label{lem:count_quasi-arcs}
    Let $X$ be a generalized sin(1/x)-type continuum. Then $X$ contains at most countably many oscillatory quasi-arcs  without ancestors.
\end{lem}
\begin{proof}
    Suppose $X$ is a generalized sin(1/x)-type continuum with an uncountable family of pairwise disjoint oscillatory quasi-arcs  without ancestors $\{L_\alpha\}_{ \alpha \in A}$.  Fix a parameterization
    $\phi_\alpha \colon [0,+\infty) \rightarrow L_\alpha\subset X$
    of $L_\alpha$. Denote $\mathring{L}_\alpha=L_\alpha \setminus \phi_\alpha(0) $.  Let $\phi \colon X\to Y$ be the map provided by the definition of a generalized sin(1/x)-type continuum.
    As points with nondegenerate preimage have to be nowhere dense in $Y$, $\phi$ is a bijection on each $\mathring{L}_\alpha$, hence
    disjoint quasi-arcs do not map by $\phi$ onto the same arcs in $Y$. Sets $\mathring{L}_\alpha$ may intersect, however, it is not hard to see that the intersections are allowed only at the branching points of $Y$, so there are  only finitely many $L_\alpha$ such that $\phi(L_\alpha)$ contains a branching point. By removing the indexes of these arcs from $A$ (should there be any) we get a family of disjoint open arcs in $G$ with cardinality no smaller than that of $A$ (in particular uncountable).  But $Y$ is a topological graph, and hence it does not allow an uncountable family of open connected disjoint subsets. It is a contradiction, completing the proof.
\end{proof}

\begin{rem}\label{rem:48osci}
Generalized sin(1/x)-type continuum can contain only countably many oscillatory quasi-arcs but may contain uncountably many tranches.
It means that the situation where a tranche is a limit set of oscillatory quasi-arc (e.g. like in the Warsaw circle) is to some extent special.
\end{rem}

In \cite[Example 26]{MR3272777} the authors provided an example of a generalized sin(1/x)-type continuum $X$ with map $\phi\colon X\to [0,1]$
such that there is a Cantor set $Q \subset [0,1]$ with nondegenerate $\phi^{-1}(y)$ for every $y\in Q$. It is a particular example of the situation described in Remark~\ref{rem:48osci}.

Later, in Example~\ref{exmp:dense_no_arcs} we show that even stronger extreme is possible. We will construct a generalized sin(1/x)-type continuum with a dense set of tranches and without an oscillatory quasi-arc (in fact, $X$ does not contain any arcs).

\begin{lem}\label{lem:4.3}
Suppose that $X$ is a tranched graph and $L \subset X$ is an oscillatory quasi-arc. Then $\omega(L)$ is a subset of a tranche.
\end{lem}
\begin{proof}
 As $\omega(L)$ is a nowhere dense subcontinuum of $X$, the result follows from Remark~\ref{rem:collapse}
\end{proof}
    \begin{figure}

            \centering
            \includegraphics[scale=0.32]{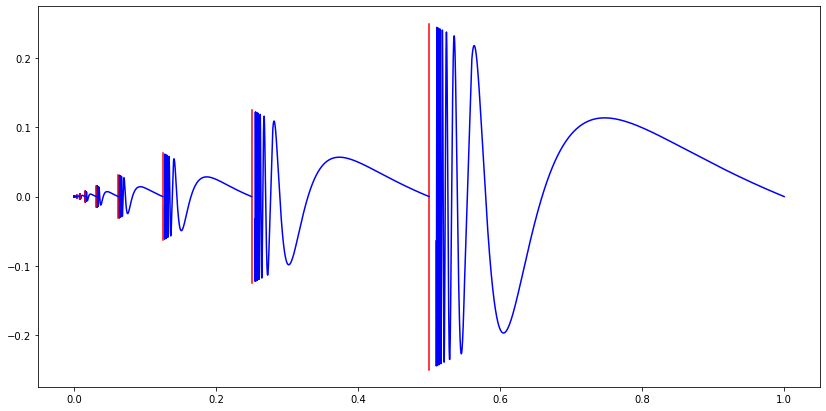}
            \caption{An exemplary generalized sin(1/x)-type continuum whose set of tranches is not closed}            \label{fig:tranche_not_closed}
        \end{figure}   

For the construction in Example~\ref{exmp:infinite_order_quasi_arc} we will use a continuum homeomorphic to the classical sin(1/x)-continuum. However, note that we slightly modified the classical geometry of this object.
Although we will still require
that its limit set is  
$\{0\} \times [0,1]$, we also demand that no point of the associated oscillatory quasi-arc 
intersects the set $[0,1]\times \{0\}$ and intersects $[0,1]\times \{1\}$
only at its endpoint (see Figure~\ref{fig:map_f}). 
Clearly, we can view this quasi-arc as a graph of a continuous map from the set $(0,1]$ onto itself, allowing the following inductive procedure.

\begin{exmp}
\label{exmp:infinite_order_quasi_arc}
    Let    $f\colon (0,1] \rightarrow (0,1]$  be a continuous surjective mapping such that $f(x)=1$ if and only if $x=1$, $f(x)$  approaches $0$ like classical sin(1/x) continuum and $f(x) \neq 0$ for any $x \in (0,1]$, 
    see Figure~\ref{fig:map_f}.  Strictly speaking, we put $f(t)= 0.5((1-t)\sin(1/t)+1)$ for $t\leq 0.7$ and for $t\geq 0.7$ the map $f$ is affine with $f(1)=1$.
    
\end{exmp}

\begin{figure}[ht]
            \centering
            \includegraphics[scale=0.35]{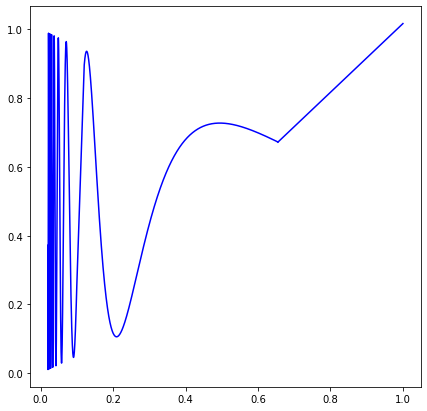}
            \caption{The map $f \colon (0,1] \rightarrow (0,1]$ from Example~\ref{exmp:infinite_order_quasi_arc}}
            \label{fig:map_f}
        \end{figure}
        
 Let us now construct following continua:
\begin{eqnarray}
A_0&=&\{(x,0,0,\ldots): x\in (0,1]\} \cup \{0\}^\infty\nonumber \\
A_1&=&\{(x,f(x),0,\ldots): x\in (0,1]\}\cup \{(0,x,0,\ldots): x\in (0,1]\} \cup \{0\}^\infty\nonumber\\
&\vdots&\nonumber\\
A_n&=&\{(x,f(x),\ldots,f^n(x),0,\ldots): x\in (0,1]\}\cup \theta(A_{n-1})\label{def:An}\\
&\vdots&\nonumber
\end{eqnarray}
where $ \theta \colon [0,1]^\N \rightarrow [0,1]^\N$ is the right shift defined by $\theta((x_0,x_1, \ldots))=(0, x_0, x_1,\ldots)$.

It is easy to see that the sequence $\{A_n\}_{n=1}^\infty$ converges, as the difference between the $n-th$ and $(n+1)th$ elements only appears in the $(n+1)th$ coordinate. Now let $A=\lim_{n \ra \infty} A_n$, which
equivalently means that:
\begin{equation}
    \label{eq:A}
A= \bigcup_{n=0}^\infty\theta^n(\{x,f(x),\ldots,f^n(x),\ldots): x\in (0,1]
\})\cup \{0\}^\infty. 
\end{equation}
In what follows, we will take a closer look at the properties of the set $A$. It will provide us with an intuition before proceeding with more complicated
examples of similar flavor.

\begin{lem}\label{lem:An-sin1x}
    Each of the continua $A_n$ defined by \eqref{def:An} is a generalized sin(1/x)-type continuum and arc-like.
\end{lem}
\begin{proof}

 It is enough to use   Remark~\ref{rem:arclike_are_classw} and Remark~\ref{rem:classw_is_hereditary}.
The continuum $A_1$ is an interval, hence arc-like and an element of $\ClassW{}$. Observe that $A_{n+1}$ is a closure of an  oscillatory quasi-arc $L_{n+1}$ with $\omega(L_{n+1})=\theta(A_n)$. Since $\theta$ is a homeomorphism between $A_n$ and $\theta(A_n)$, we see that $\theta(A_n)$ is arc-like, and so $A_{n+1}=L_{n+1} \cup \omega(L_{n+1})$ is a generalized sin(1/x)-type continuum, as $\omega(L_{n+1}) \in \ClassW$ . It is elementary to check that since $\theta(A_{n+1})$ is arc-like, so is $A_{n+1}$.
\end{proof}
\begin{figure}[ht]
            \centering
            \includegraphics[scale=0.27]{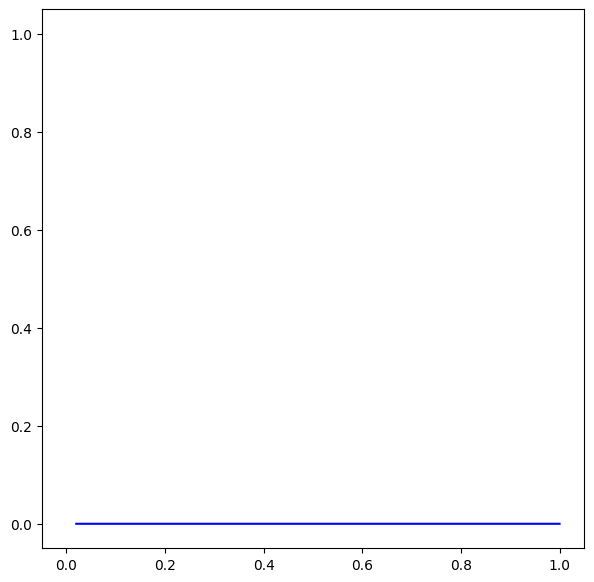}
            \includegraphics[scale=0.27]{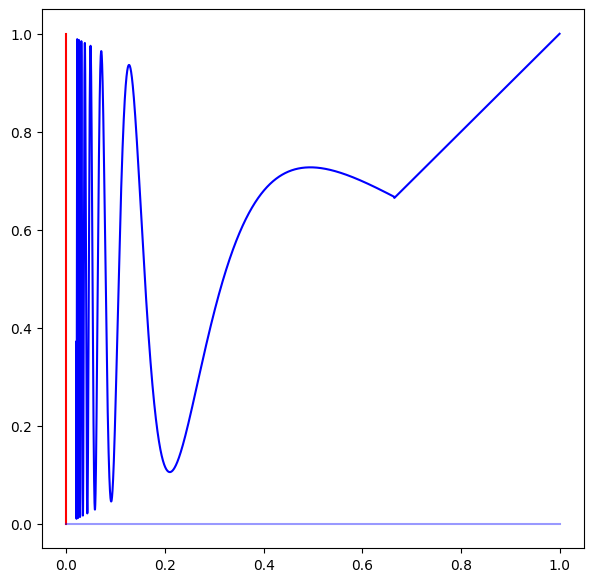}
            \includegraphics[scale=0.28]{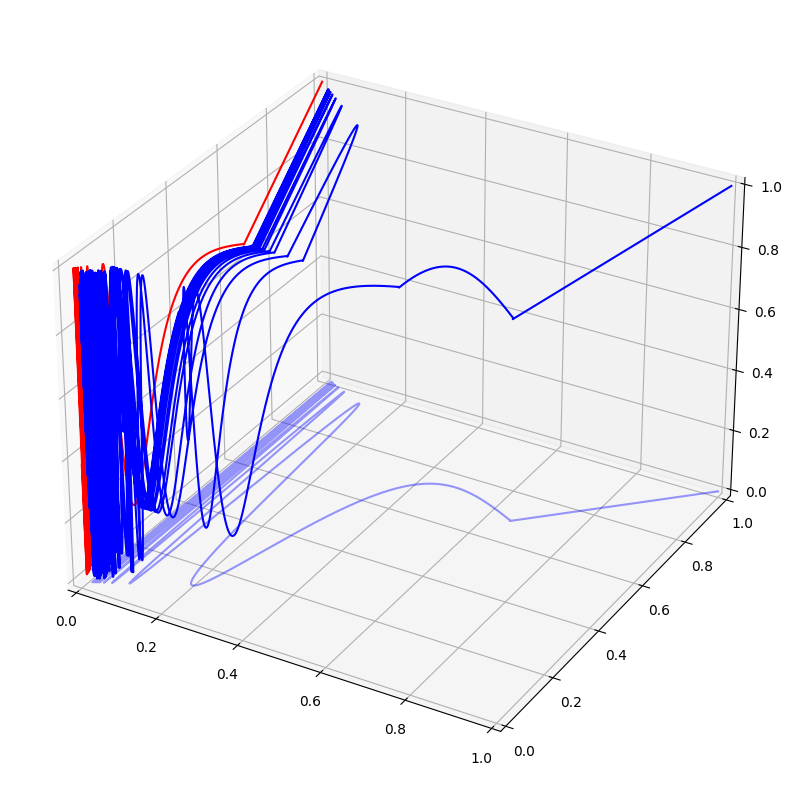}
            \caption{Projections of continua  $A_0$, $A_1$, $A_2$ defined by \eqref{def:An}.}
            \label{fig:contA3}
        \end{figure}

\begin{lem}
\label{lem:infite_depth}
    The set $A$ defined by \eqref{eq:A} is a generalized sin(1/x)-type continuum.
\end{lem}
\begin{proof}
It is clear that $A$ is a compact set. Now suppose that it is not connected. Then there are closed and disjoint sets $U,V$ such that $A=U \cup V$ and as a consequence there is a coordinate $n$ such that projections $U_n$ and $V_n$ are also disjoint.
The projection of $A$ onto the $n$-th coordinate is the set
$$\bigcup_{i=0}^n \{f^i(x) : x\in (0,1]\}\cup \{0\}=[0,1]
= U_n\cup V_n
$$ 
which is a contradiction.

    Let $Y=[0,1]$ and let $\phi\colon A \rightarrow Y$ be the projection onto  the first coordinate. 
 Observe that $\phi^{-1}(\{0\})=\{x\in A: x_0=0\}=\theta(A)$, hence it is  a compact connected set. For $x\in (0,1]$ each fiber $\phi^{-1}(x)$ is a singleton. This shows that $\phi$ is a monotone map.      
    
    Let $Z$ be a subcontinuum of $\phi^{-1}(0)$. 
    Note that the projection of $A$ onto the first $n+1$-coordinates is the same as the projection of $A_n$  onto these coordinates.
    Take $n$ such that if $y,z$ satisfy $y_i=z_i$ for $i=0,\ldots,n$ then $d(y,z)<\varepsilon/4$.
    Let $Z_n$ be the projection of $Z$ onto first $n+1$ coordinates with $0$ on coordinates with index $i>n$. In other words, $Z_n$ is the projection of $Z$ into $A_n$.
    Let $\varphi_n(x)=(x,f(x),\ldots,f^n(x),0,0,\ldots)\in A_n$
    and $\varphi(x)=(x,f(x),\ldots,f^n(x),\ldots)\in A$, defined for $x\in [0,1)$,
    be oscillatory quasi-arcs.
    By Lemma~\ref{lem:An-sin1x} there are $a,b$ such that $d_H(\varphi_n([a,b]),Z_n)<\varepsilon/4$.
    But $d_H(\varphi_n([a,b]),\varphi([a,b]))<\varepsilon/4$,
    hence $d_H(\varphi([a,b]),Z)<3\varepsilon/4<\varepsilon$ completing the proof.
\end{proof}
 Let us observe that $A$ has a kind of fractal self-similar structure. Strictly speaking, we can construct a sequence $A=X_0 \supset X_1 \supset \ldots \supset X_n$ of any finite length such that $X_{k+1}$ is a tranche of $X_k$ for $k=0, \ldots , n-1$ and all of the spaces $X_k$ are homeomorphic to $A$.  The construction presented in Example~\ref{exmp:infinite_order_quasi_arc} is also a natural way of constructing $\infty$-order oscillatory quasi-arc. This gives us another difference between quasi-graphs and generalized sin(1/x)-type continua.

\begin{rem}
    In general, generalized sin(1/x)-type continua may contain an $\infty$-order oscillatory quasi-arcs as subsets.
\end{rem}

To construct the continuum $A$ defined by \eqref{eq:A} we set a particular geometric representation of the sin(1/x)-continuum, whose only intersection with line $[0,1] \times \{0\}$ was allowed for the limit set of the sin(1/x) curve. This allowed us to lift this continuum consecutively to higher dimensions, keeping only one tranche each time, hence maintaining a general structure of oscillatory quasi-arc.

In Example~\ref{exmp:dense_no_arcs} we are going to give another geometric representation of a sin (1/x)-type continuum, but this time, for symmetry  in the construction, we use a two-sided version of the sin(1/x)-continuum (see Figure~\ref{fig:double}). If we discard the limit sets of the sin(1/x) curve, the remaining set can be parameterized as a graph of a function. The main difficulty in repeating this construction (and in providing an accessible description of the resulting space) is caused by the placement of the curve in the space. Namely, several points that map to $0$ or $1$, and so we cannot iterate the map on them any further (in contrast to Example~\ref{exmp:infinite_order_quasi_arc} where there was only one problematic point). This significantly increases the complexity of the construction (and the resulting space). In contrast to Example~\ref{exmp:infinite_order_quasi_arc}, instead of maintaining one tranche, in each step of the construction, we produce infinitely many new tranches.

\begin{exmp}
\label{exmp:dense_no_arcs}
First define auxiliary sequence $a_{-k}=\frac{1}{(k+2)^2}$ for $k \in\N $ and $a_k=1-\frac1{(k+2)^2}$ for $k>0$. Notice that intervals $I_n=[a_n,a_{n+1}]$ give us a decomposition of $(0,1)$, $|I_n|\leq\frac12$ for all $n \in \Z$ and $|I_n|=\frac12$ if and only if $n=0$. On each interval $I_n$ we plot a tent map with slope $\lambda_n=2\frac1{|I_n|}$, notice that $\lambda_n \geq 4$. Denote by $X$ the closure of the graphs of tent maps (see Figure~\ref{fig:double}).

\begin{figure}[ht]
    \centering
    \includegraphics[scale=0.25]{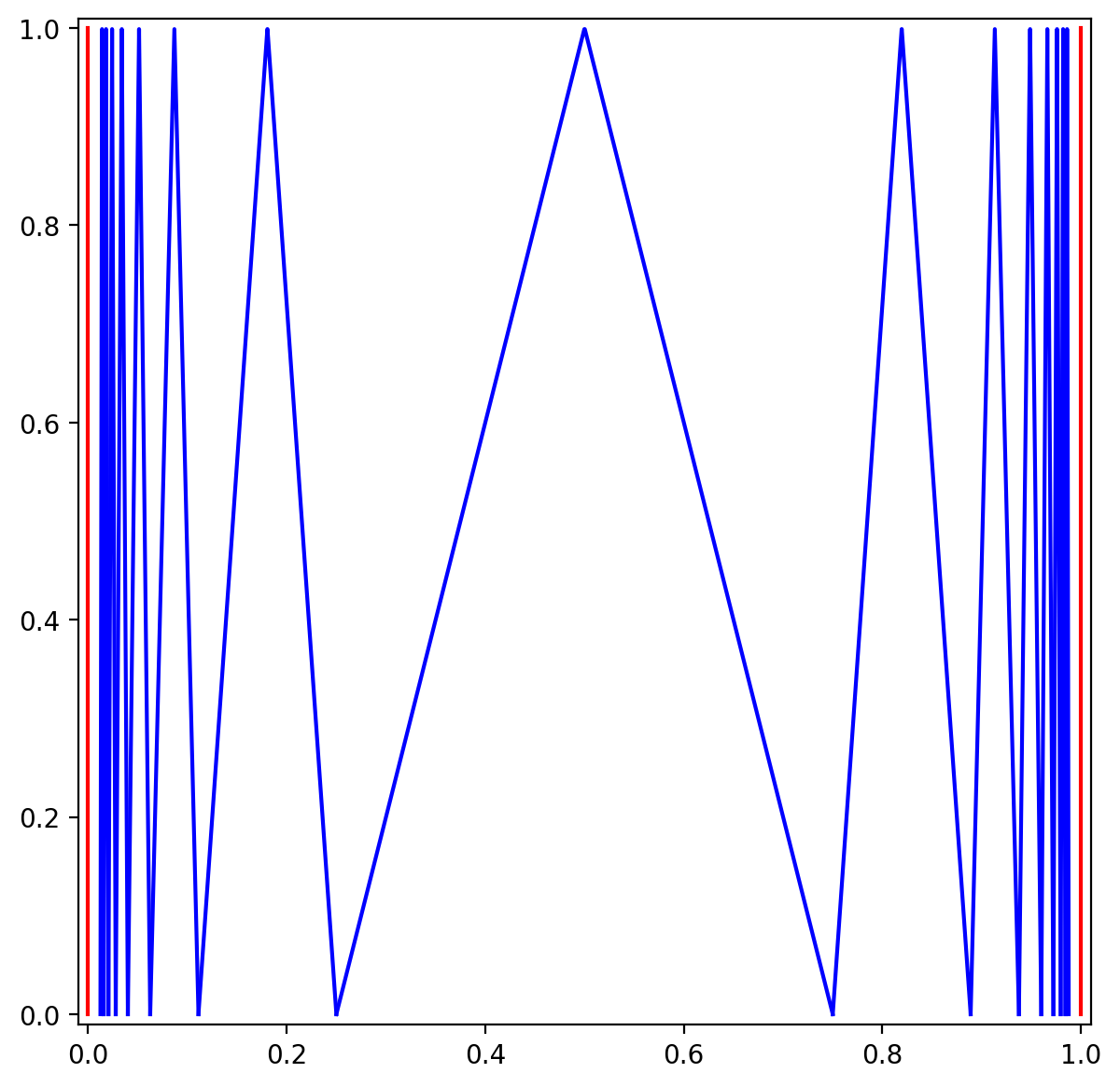}
    \includegraphics[scale=0.4]{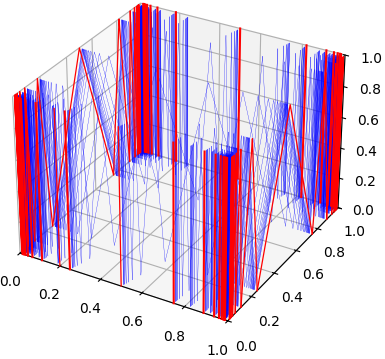}
    \caption{The continuum $X$ and projection of continuum $X_2$ from Example \ref{exmp:dense_no_arcs}. Points where $X_2$ is not locally arcwise connected are marked in red. Roughly speaking they appear at the faces of the boundary cube; on two faces these sets are copy of $X$; at another two adjacent faces they appear as vertical lines exactly at extrema of blue curve defining $X$. }
    \label{fig:double}
\end{figure}
It is easy to check that $X$ is a generalized sin(1/x)-type continuum.
We construct the following sequence of continua embedded in the Hilbert cube $\mathcal{H}=[0,1]^\N$:
\begin{eqnarray*}
	X_0&=&\{(x_0,0,0,\ldots): x_0\in [0,1]\}\\
	X_1&=&\{(x_0,x_1,0,\ldots): x_0\in [0,1],(x_0,x_1) \in X \}\\
	&\vdots&\\
	X_n&=&\{(x_0,x_1,\ldots,x_n,0,\ldots): x_0\in [0,1], \forall i=1,...,n \quad (x_{i-1},x_i) \in X \}\\
	&\vdots&\\
\end{eqnarray*}
\end{exmp}

The ultimate goal of our construction is to prove that the sequence $\{X_n\}_{n=1}^\infty$ converges in the Hausdorff metric to a generalized sin(1/x)-type continuum. Before we proceed further, we describe the topological structure of continua $X_n$. 
Let $\psi \colon X \ra [0,1]$ and $\psi_n \colon X_n \ra [0,1]$ be projections onto the first coordinate. 
The following lemma describes the subcontinua of the continuum $X_n$:
\begin{lem}
\label{lem:no_arcs_finite_tranche_subcontinuum}
    Let $Z \subset X_{n}$ be a subcontinuum with $\psi_n(Z)=y$ for some $y \in [0,1]$. Then $Z= \{y\} \times Z_0$ and $Z_0$ is a subcontinuum of $X_{n-1}$.
\end{lem}
\begin{proof}
    As $\psi_n \colon X_n \ra [0,1] $ is a projection, obviously $Z= \{y\} \times Z_0$ for some set $Z_0$. Denote by $\sigma$ the standard left shift, meaning $\sigma ( (x_0,x_1,\ldots) )= (x_1,x_2,\ldots)$.
    Then $Z_0=\sigma(Z)$ and so is a continuum, since $\sigma$ is a continuous map.
    But the relation $X$ is surjective, and hence by definition $\sigma(X_n)=X_{n-1}$, completing the proof.
\end{proof}

Now we are ready to prove the following:
\begin{lem}
\label{lem:no_arcs_finite_sin}

	Each continuum $X_k$, $k=0,1,\ldots$ is a generalized sin(1/x)-type continuum with the associated map $\psi_k \colon X_k \ra [0,1]$ (that is, $\psi_k$ satisfies the definition requirements). 
\end{lem}
\begin{proof}
	 We will proceed by induction.  By the definition, $X_0$ is an arc and $X_1$ is homeomorphic to $X$, so both are a generalized sin(1/x)-type continua.   As we noticed before, $X$ can be viewed as the closure of the graph of a function, say 
  $f \colon (0,1) \ra [0,1]$
  
  Now assume we already proved that $X_{k-1}$ is a generalized sin(1/x)-type continuum for some $k\geq 2$.
    Choose a point $x_0\in [0,1]$ such that $\psi^{-1}_k(x_0)$ is nondegenerate and fix any point $x=(x_0, \ldots,x_{k},0, \ldots)\in \psi^{-1}_k(x_0)\subset X_k$. We claim that for every $\epsilon>0$ there is an element $x_\epsilon \in X_k$ defining a singleton fiber and such that $d(x,x_\epsilon)<\epsilon$. 
    
     Assume that $x_0=0$ and $x_1,\ldots, x_k\not\in \{0,1\}$, then using the fact that $X$ is a generalized sin(1/x)-type continuum, for every $\epsilon >0$ there is a point $(z,x_1) \in X$ with $0<z<\epsilon$. By our construction $x_\epsilon = (z,x_1,\ldots,x_{k},0, \ldots) \in X_k$ and obviously $d(x, x_\epsilon)< \epsilon$. For $x_0=1$ the process is completely analogous.  If there were positions $x_i\in \{0,1\}$ we can perform the above construction of $x_\eps$ in steps, starting with largest $i\leq k$ such that $x_i\in \{0,1\}$ and perform the above modification starting with this position, next move to smaller $i$, eventually finishing at $x_0$.

    Suppose now $x_0 \not\in \{0,1\}$, choose $\delta<\epsilon/2$ such that $J=(x_0 - \delta, x_0+\delta) \subset (0,1)$ and $f(J) \subset [x_1 - \epsilon/2, x_1 + \epsilon/2]$. By induction hypothesis there is a point $\tilde{x}=(\tilde{x}_1, \ldots, \tilde{x}_{k} , 0,0,\ldots)$ such that $d((  x_1\ldots,x_{k},0, \ldots),\tilde{x})<\epsilon/2$  and $\tilde x_1\in f(J)$, which is possible since $f(J)\ni x_1$ is an open set for small $\delta$. We choose $\tilde{x}_0 \in J$ such that $f(\tilde{x}_0)=\tilde{x}_1$     and we set $x_\epsilon=(\tilde{x}_0, \ldots, \tilde{x}_{k})$. It follows easily that $d(x,x_\epsilon) <\epsilon$.
    We get that the set of degenerate fibers is dense in $X_k$, so the claim is proved.

    Next we are going to prove the approximation property.     First we claim that if $Z \subset X_k$  is a continuum such that $\psi_k(Z)=[a,b]$ for a nondegenerate interval $[a,b]\subset [0,1]$, then $Z=\psi^{-1}_k([a,b])$. Denote by $D_k \subset [a,b]$ the set of points with degenerate preimages under $\psi_k$. We proved already that points from singleton fibers are dense in $X_k$, and the argument from the proof easily leads to  $\overline{\psi^{-1}_k(D)}=\psi^{-1}_k([a,b])$.
    On the other hand, points in the set $\psi^{-1}_k(D)$ map injectively to $[a,b]$, meaning $\psi^{-1}_k(D)\subset Z$, and so $Z$ is dense in $\psi^{-1}_k([a,b])$. But $Z$ is closed, so $Z=\psi^{-1}_k([a,b])$ proving that the claim holds.

     Fix any $y\in [0,1]$ and  $\epsilon>0$. If $\psi^{-1}_{k}(y)$ is a singleton, the approximation property is trivially satisfied. Suppose that $\psi^{-1}_{k}(y)$ is nondegenerate and fix any subcontinuum $Y_0$ of $\psi^{-1}_{k}(y)$.  
     Lemma~\ref{lem:no_arcs_finite_tranche_subcontinuum} implies that $Y_0=\{y\} \times Z_0$, for some subcontinuum $Z_0$ of  $X_{k-1}$. Now either  $ \psi_{k-1}(Z_0)=[a,b]$ for $a\neq b$, or $Z_0$ is a subset of a tranche in $X_{k-1}$. 
    Assume that the first possibility holds.
    Choose $[\alpha,\beta] \subset [0,1]$ such that $\dist(y,[\alpha,\beta])<\epsilon$ and $R([\alpha,\beta])=[a,b]$, where  $R(J)=\{z \in [0,1]: (x,z) \in X, x\in J \}$. Then as we showed before, there is an unique continuum $A=\phi_{k}^{-1}([\alpha,\beta])$ that projects to $[\alpha,\beta]$, and so $\sigma(A)=Z_0$, $\sigma$ being the left shift, defined as in Lemma~\ref{lem:no_arcs_finite_tranche_subcontinuum}. That gives us
    $$d_H(Y_0,\psi^{-1}_{k}([\alpha,\beta]))=\dist(y,[\alpha,\beta])<\epsilon.$$
    Suppose now that $Z_0$ is a subset of a tranche in $X_{k-1}$. By induction hypothesis, there is $[a,b] \subset [0,1]$ with
    $$d_H(\psi^{-1}_{k-1}([a,b]),Z_0) <\epsilon/2$$
    Again, we choose $[\alpha,\beta] \subset [0,1]$ such that $\dist(y,[\alpha,\beta])<\epsilon/2$ and $R([\alpha,\beta])=[a,b]$, giving us 
    $$d_H(Y_0,\psi^{-1}_{k}([\alpha,\beta]))<\epsilon.$$
    \end{proof}

It is easy to see that the sequence $\{X_n\}_{n=1}^\infty$ defined in Example~\ref{exmp:dense_no_arcs} converges, as the changes we make occur on higher dimensions in each step. Denote $\widehat{X}=\lim\limits_{n \ra \infty} X_n$ and let $\hat{\psi} \colon \widehat{X} \ra [0,1]$ be the projection onto the first coordinate. Construction of $\widehat{X}$ is similar to that of an inverse limit, and is sometimes called the infinite Mahavier product of $X$, and has generated some attention recently (see \cite{MR3179790}). We are going to show that $\widehat{X}$ is a generalized sin(1/x)-type continuum. The first step is to show that all fibers are nowhere dense, which is $(i)$ in Definition~\ref{def_sin1/x} and will be used to prove point $(ii)$ in that definition. 

\begin{lem}
\label{lem:fibers:nd}
     The fiber $\hat{\phi}^{-1}(y)$ is nowhere dense in $\hat{X}$ for any $y \in [0,1]$.
\end{lem}
\begin{proof}

    If the fiber is a singleton, the result is obvious. Assume otherwise and suppose that there is an open set $U \subset \hat{\phi}^{-1}(y)$. 
    There is $n$ such that $U$ has a nondegenerate projection onto $n$-th coordinate. But then $\psi_n^{-1}(y)$ has nonempty interior in $X_n$, which is a contradiction by Lemma~\ref{lem:no_arcs_finite_sin}.
\end{proof}

\begin{lem}
\label{lem:tranche:count}
\label{lem:tranche_baire}

Let $N=[0,1] \backslash D$, where  $D \subset [0,1]$ is the set of points with a degenerate preimage. The set $\hat{\psi}^{-1}(N)$ is dense in $\widehat{X}$, countable, and meager.
As a consequence, the set
$\hat{\psi}^{-1}(D)$ is residual.
\end{lem}
\begin{proof}
    Let us observe that  $X_1$ has two tranches  $\psi_1$-preimages of points $0$ and $1$.  It follows that the distance between the projection of tranches by $\psi_1$ is $1$. By our construction, in $X_2$ the longest open  subintervals of $[0,1]$ with all points having degenerate preimages under $\psi_2$ are $(a_0,\frac{a_0 +a_1}{2})$ and $(\frac{a_0 +a_1}{2},a_1)$, both with length $\frac1{\lambda_0}$ , which are pieces of monotonicity of the tent map with the smallest possible slope used to define $X$. These intervals get stretched by a factor of $\lambda_0$, so it is easy to see that the longest interval without a degenerate preimage under $\psi_3$ has length $\frac1{\lambda_0}\frac1{\lambda_0}= \frac1{\lambda_0^2}$. Continuing inductively we get that the longest interval without a nondegenerate $\psi_{n+1}$-preimage has length $\frac{1}{\lambda_0^n}$. By passing to the limit, we get that for any  $y \in [0,1]$ and any $\epsilon>0$ set $(y-\epsilon,y+\epsilon) \cap N$ is nonempty and so $N$ is dense in $[0,1]$. 
     
     We are going to repeat the argument from the proof of Lemma~\ref{lem:no_arcs_finite_sin} to show that the set $\hat{\psi}^{-1}(N) $ is also dense. Take any point $y$ with a degenerate preimage. By the density of $N$, there is a sequence $\{ y_n\}_{n=1}^\infty\subset N$ 
     such that $y_n\to y$. For each $i$ fix an element
      $x_i \in \hat{\psi}^{-1}(y_i)$  
      and assume by compactness (going to a subsequence if necessary) that
 the limit $x_n\to x$ exists. But then $\hat\psi(x_n)=y_n\to y=\hat\psi(x)$
 and therefore
 $
      \hat \psi^{-1}(y)=\{x\}$. 
      Indeed $\hat{\psi}^{-1}(N)$ is dense in $X$ as claimed.
    \par
    The set of points with a nondegenerate preimage under $\psi_k$ is countable for all $k$ and $N$ is a countable  union of such sets, therefore $N$ is countable.

By Lemma~$\ref{lem:fibers:nd}$ 
the set  $\hat{\psi}^{-1}(N)$ is a meager set. As $\hat{\psi}^{-1}(D)$ is its  
complement in  $\widehat{X}$, $\hat{\psi}^{-1}(D)$ is a residual set. 
\end{proof}

When we constructed continua $X_k$, their trenches were homeomorphic to $X_{k-1}$. 
This extends onto
$\widehat{X}$.
\begin{lem}
\label{lem:homeo}
    Each fiber $\hat{\psi}^{-1}(y) \subset \widehat{X}$ is either a singleton or is homeomorphic to $\widehat{X}$
\end{lem}
\begin{proof}
    Let $y \in [0,1]$ be such that $\hat{\psi}^{-1}(y)$ is not a singleton. It means that there is  $k \in \N$ such that $\psi_k^{-1}(y)$ is a tranche of $X_k$. Fix the smallest $k$ with this property, and note that $\psi_{k-1}^{-1}(y)$ is not a tranche of $X_{k-1}$. Therefore the first $k-1$ coordinates of any element in $\hat{\psi}^{-1}(y)$ are uniquely determined and belong to $(0,1)$. Using the symmetry of $X$, without loss of generality we may assume that the $k$-th coordinate is equal to $0$ for every element of $\hat{\phi}^{-1}(y)$. Then:  
    $$\hat{\psi}^{-1}(y)=\{(y_0,y_1, \ldots, y_{k-1},0, y_{k+1},\ldots), (y_{i-1},y_i) \subset X \}$$
    for $y_0=y$. Set of points $z$ for which $(0,z) \in X$ is equal to the interval $[0,1]$, so:
    $$\hat{\psi}^{-1}(y)=\{(y_0,y_1, \ldots, y_{k-1},0, z_0, z_1,\ldots), (y_{i-1},y_i) \in X, ~(z_{i-1},z_i) \in X, ~z_0 \in [0,1]\}$$
     Observe that the map $\tau_k(x)=\sigma^k(x)$
      is invertible on $\hat{\psi}^{-1}(y)$ and the set $\tau_k(\hat{\psi}^{-1}(y))$ is equal to $\widehat{X}$. Indeed $\hat{\psi}^{-1}(y)$ and $\hat X$ are homeomorphic.
\end{proof}
We can now prove the theorem revealing the main property of our construction.
\begin{thm}
\label{thm:hat_sin}
     The set $\widehat{X}$ is a generalized sin(1/x)-type continuum
\end{thm}
\begin{proof}
    By Lemma~$\ref{lem:homeo}$, each fiber $\hat{\psi}^{-1}(y)$ is a singleton or is homeomorphic to $\widehat{X}$. In both cases, it is a connected set, meaning $\hat{\psi}$ is monotone.
    By Lemma~\ref{lem:tranche_baire} and the Baire Category Theorem, the set of degenerate fibers is dense in $\widehat{X}$. It remains to show the approximation property from the definition of generalized sin(1/x)-type continua.
    
 Fix any $y\in [0,1]$ and $\epsilon>0$. If the preimage of $y$ is degenerate, there is nothing to prove, so assume that $\hat{\psi}^{-1}(y)$ is nondegenerate and fix a subcontinuum  $Y_0 \subset \hat{\psi}^{-1}(y)$.
 By the construction, there is a natural number $k$ such that $\psi^{-1}_k(y)$ is a tranche of $X_k$. Let $\pi_k$ be the projection:
 $$\pi_k \colon \widehat{X}\ni (x_0,x_1, \ldots) \mapsto (x_0,x_1, \ldots, x_k,0,0,0, \ldots) \in  X_k$$
 Let us choose  $k$ large enough so that $Z_0=\pi_k(Y_0)$  is a subcontinuum of a tranche of $X_k$, and any set $V \subset \mathcal{H}$ satisfies:
 $$d_H(V,\pi_k(V))  <\epsilon/4.
 $$
By Lemma~\ref{lem:no_arcs_finite_sin} we obtain that $X_k$ is a generalized sin(1/x)-type continuum, so there is an arc  $[a,b] \subset [0,1]$ such that 
$B=\psi_k^{-1}([a,b])$ approximates $Z_0$:
 $$d_H(B,Z_0) < \epsilon/4$$
Let $A$ be an extension of $B$ to $\hat X$ 
 given by the following formula:
 $$A=\{ (x_0,x_1, \ldots,x_k,x_{k+1}, \ldots) \in \hat{X}:~ (x_1, \ldots, x_k,0, \ldots) \in B \}=\pi_k^{-1}( B).$$
Such an infinite extension is nonempty, because for any element of $B$ we can consider all replacements of $0$ on $k+1$-coordinate using all the elements $(x_k,x_{k+1})\in X$ and proceed with this process inductively. We can view the process of creation of $A$ as a Cauchy sequence in the hyperspace (or use the fact that $\pi_k$ is continuous); hence, $A$ is closed. This means that 
the Hausdorff distance 

is well defined on $A$ and
 $$ d_H(A,Y_0) \leq d_H(A, B)+d_H( B,Y_0) \leq d_H(A, B) +d_H(B,Z_0) +d_H(Z_0,Y_0).$$
 To sum up, we get
 $$d_H(\hat{\psi}^{-1}([a,b]),Y_0) = d_H(A,Y_0) \leq \epsilon/4 + \epsilon/4 +\epsilon/4 < \epsilon$$
    proving the approximation property, and so $\widehat{X}$ is a generalized sin(1/x)-type continuum.
\end{proof}

\begin{thm}
     The continuum $\widehat{X}$ does not contain any nondegenerate arcs.
\end{thm}
\begin{proof}
 Assume on the contrary that there exists a nondegenerate arc $A$ in $\widehat{X}$. If $\hat{\psi}(A)=[a,b]$ for some $a \neq b \in [0,1]$, then by Lemma~\ref{lem:tranche:count} we can find $y \in [a,b]$ with a nondegenerate preimage under $\hat{\psi}$. It follows that $\psi_k^{-1}(y)$ is a tranche in  $X_k$ for some arbitrarily large $k \in \N$, and so is the limit set of an oscillatory quasi-arc.  But then the projection of $A$ onto $X_k$ is not arcwise connected, which would contradict the fact that $A$ is an arc.  This means that $A$ is a subset of a tranche of $\widehat X$ or it does not intersect any tranche, which means it is a singleton. 
 
Since $A$ is nondegenerate, there is $y_0 \in [0,1]$, such that all points $y \in A$ have $y_0$ on the first coordinate, which means they have the form $y=(y_0,x_1,x_2,\ldots)$. By our construction $\sigma(y,x_1,x_2, \ldots) \in \widehat{X}$ and the set $\{(x_1,x_2,\ldots) : (y_0,x_1,x_2,\ldots)\in A\}$ is a nondegenerate arc. By the same argumentation as before, the second coordinate is the same for all elements of $A$. By induction, for any two elements $y,\tilde{y} \in A $ and any coordinate $i$, we have $y_i=\tilde{y}_i$. It follows that $A$ is a singleton, which contradicts our assumptions. We get that $\widehat{X}$ contains no nondegenerate arcs.
\end{proof}

The well-known example of a nondegenerate continuum that does not contain any arcs is the pseudo-arc (e.g. see \cite{MR3573330}), a hereditarily indecomposable continuum.

However, we can prove that $\hat X$ has an opposite extreme property, i.e. it is hereditarily decomposable. 
It turned out that tranched graphs are always decomposable, and in hereditary tranched graphs become hereditary by their natural structure.

\begin{prop}
    Every tranched graph is decomposable; moreover, hereditary tranched graphs are hereditarily decomposable.
\end{prop}
\begin{proof}
 Let $\phi \colon X \ra Y$ be a continuous map on a topological graph $Y$ from Definition~\ref{def:tranched_graph}. Since $Y=A\cup B$
for two proper subcontinua, $X=\phi^{-1}(A)\cup \phi^{-1}(B)$
showing it decomposability. By a similar argument, every subcontinuum which is not completely contained in a tranche is decomposable.

If $X$ is a hereditary tranched graph, all nondegenerate subcontinua of $X$ are tranched graphs,
and so the result follows.

\end{proof}
 The next result shows that the controlled collapse of a subset of hereditary tranched graph leads to another space in this class of continua.

\begin{prop}
\label{prop:quitent_space_is_tranched_graph}
    Let $X$ be a hereditary tranched graph and $\phi \colon X \ra Y$ be an associated mapping. Let $\sim $ be a closed equivalence relation on $X$ with finitely many non-degenerate and connected equivalence classes preserved by $\phi$, i.e. if $p\sim q$ then $x\sim y$ for any $x\in \phi^{-1}(\phi(p))$ and $y\in \phi^{-1}(\phi(q))$. Then $X /_\sim$ is also a hereditary tranched graph. 
\end{prop}
\begin{proof}
    Define the relation $\wr$ on $Y$ by putting $y_1 \wr y_2$ if and only if there is $x_i \in \phi^{-1}(y_i)$ for $i=1,2$ such that $x_1 \sim x_2$. 
    Since $\sim$ is preserved by $\phi$ the relation $\wr$ is an equivalence relation.
    Denote $A = X /_\sim$ and $B =Y /_\wr$. First, observe that the map $\psi \colon A \ni [x]_\sim \mapsto [\phi(x)]_\wr \in B$ is well defined and monotone. As $\sim$ is a closed equivalence relation collapsing to a point at most finitely many connected sets in $Y$, obviously $A$ is a compact metric space and $B$ is a topological graph. As the set of degenerate fibers of $\phi$ was dense in $X$, the set of degenerate fibers of $\psi$ is dense in $A$.
    The same argument works for a subcontinuum of $X/_\sim$, so we find that $X /_\sim$ is a hereditary tranched graph.
\end{proof}

Right now, we have shown that there exists a generalized sin(1/x)-type continua of infinite depth (which we define in Definition~\ref{def:4.21}) and width (the number of tranches of the continuum). By Lemma~\ref{lem:quasi-graphs_are_tranched} we know that quasi-graphs are hereditary tranched graphs with finite depth and width (i.e. number of tranches). This, alongside the arcwise connectedness of quasi-graphs, gives four properties that seem to characterize quasi-graphs in the class of tranched graphs. Our goal now is to show that we can deduce that arcwise connected tranched graphs always have finite width. This will remove one necessary condition from the assumptions.

\begin{defn}\label{def:4.21}
For a hereditary tranched graph $X$, let us denote $lvl_n(X)=\{T \colon X=T_0 \supset T_1 \supset \ldots \supset T_{n-1} \supset T_n=T \text{ and } T_{k+1} \text{ is a tranche of } T_k \}$.
We will call $\sup_{n \in \N}\{n: lvl_n(X) \neq \emptyset\}$ the 
 \textit{depth} of continuum $X$.
\end{defn} Notice in particular that if $X$ is a quasi-graph, then the following numbers coincide:
\begin{enumerate} 
\item depth of $X$ considered as a hereditary tranched graph, and 
\item maximal order among quasi-arcs of $X$ considered as quasi-graph.
\end{enumerate}

The following lemma is a method of reducing complexity of hereditary tranched graph, by removing oscillatory quasi-arcs ``from the outside in'' in such a way that the modified space remains in this class.
This provides a method for reducing the complexity of the hereditary tranched graph. 
Notice that if oscillatory quasi-arc was contained in the limit set of other oscillatory quasi-arc, then it cannot be directly removed since it would make the space no longer closed. It motivates the assumption of not having ancestors in the following lemma, making the procedure of removal of quasi-arcs partially ordered in some sense. 
\begin{lem}
\label{lem:regular_finite_delete_quasi_arc}
    Let $X$ be a hereditary tranched graph with a finite set of tranches, and let $L=\varphi([0,\infty)) \subset X$ be an oscillatory quasi-arc without ancestors. Then there is $M \in \N$ such that $X \backslash \varphi((M,\infty))$ is a hereditary tranched graph. If $X$ is arcwise connected, then $X \backslash \varphi((M,\infty))$ is also arcwise connected. 
\end{lem}
\begin{proof}
     Let $X$ be a hereditary tranched graph with finite set of tranches, $\phi \colon X \ra Y$ be the continuous map from Definition~\ref{def:tranched_graph}, $L$ be an oscillatory quasi-arc, with $\omega(L)\subset T$ for some tranche $T \subset X$. If $L$ contains a branching point of $X$ or intersects any tranche, we choose $M$ large enough such that quasi-arc $\varphi((M,\infty))$ does not have these properties, which is possible by the following argument. Since $L$  is without ancestors, it maps one to one to a topological graph $Y$, hence can only contain finitely many branching points of $X$. As such, there are finitely many ,,bad'' points in $L$, which we can get rid of by shortening the quasi-arc.  Therefore, let us assume that we made the above modification when necessary and put $  \tilde{L}=\varphi((M,\infty))$.  This means that the set $T \cup  \tilde{L}$ is not arcwise connected. 
     
     Therefore, if $X$ is arcwise connected then $X_{L}=X \backslash \varphi((M, \infty)) $ is arcwise connected, because arc connecting any two points in $X_L$ may avoid intersecting $\tilde L$. Let us denote $Y_{L}=\phi(X \backslash \varphi((M, \infty))$. By our assumptions $\tilde L \cap T = \emptyset$, so  $\phi(\varphi((M,\infty))$ is an open arc, meaning $Y_L$ is a topological graph. As $X$ was a hereditary tranched graph with finite set of tranches, by Lemma~\ref{lem:finite_tranches} all of the tranches are unions of limit sets of oscillatory quasi-arcs.
     
     If there exist quasi-arcs  $K_1, \ldots, K_m \subset X_L$ such that $\bigcup_{i=1}^m\omega(K_i)=T$, then $X_L$ is a tranched graph with associated mapping $\phi_L=\phi |_ {X_L}$ and $lvl_1(X)=lvl_1(X_L)$.
    
 Assume now the other possibility that for any quasi-arcs $K_1, \ldots, K_m \subset X_L$ we have $\bigcup_{i=1}^m\omega(K_i)\neq T$. We present the sketch of the following procedure in Figure~\ref{fig:removing_arcs}.
 
From hereditarity of $X$ we know that $T$ is a tranched graph, so let $\eta\colon T\to Z$ be an associated continuous map, where $Z$ is a topological graph. Denote $\Omega=\eta^{-1}(\eta(\omega(X_L)\cap T))$ and let $\sim$ be the equivalence relation such that $a \sim b$ if $a=b$ or there exists a connected component $\Lambda$ of $\Omega$ such that $a,b \in \Lambda$.
We extend $\sim$ trivially from $T$ to $X_L$ by adding to it the diagonal in $X_L\times X_L$, so we can view $\sim$ also as relation on $X_L$.
 Clearly the relation $\sim$ is preserved by $\eta$, hence, by Proposition~\ref{prop:quitent_space_is_tranched_graph}, we get that $T/_\sim$ is a hereditary tranched graph. Let $\phi_T$ be the map from the definition of tranched graph for $T/_\sim$.

Let us write $X_1$ for the closure of $(X_L/_\sim) \backslash (T/_\sim)$ in $X_L/_\sim$ and $Y_T$ for $\phi_T(T/_\sim)$. Notice that $X_1$ has a nonempty intersection with $T/_\sim$ and that it is a finite collection of topological graphs. Define an equivalence relation $\approx$ on $X_1 \cup Y_T$ such that $x \approx y$ if and only if (i) $x=y$ or (ii) $x \in T/_\sim$, $y \in Y_T$ and $y= \phi_T(x)$  or (iii) $x,y \in T/_\sim$ and  $\phi_T(x)=\phi_T(y)$. Observe that there are only finitely many equivalence classes of $\approx$ and both $X_1$ and $Y_T$ are topological graphs, hence $Y_L= (X_1 \cup Y_T)/_\approx$ is a topological graph. 
We define map $\phi_L \colon X_L \ra Y_L$ by putting $\phi_L(x)=(\pi_\approx \circ \pi_\sim)(x)$ for $x \in X_L \backslash T$ and $\phi_L(x) = (\pi_\approx \circ \phi_T \circ \pi_\sim)(x) $ for $x \in T$, where $\pi_\sim$ and $\pi_\approx$ are natural projections of relations $\sim$ and $\approx$ respectively.

 We claim that all the fibers of $\phi_L$ are nowhere dense. Consider $y$ such that $\phi_L^{-1}(y)$ is nondegenerate and assume that there is an open set $U\subset \phi_L^{-1}(y) \subset X_L$. It is not possible when $U\cap (X_L\setminus T)\neq \emptyset$ because then $\phi$ has fiber which is not nowhere dense. Therefore we may assume that $U\subset T$. But then we must have $U\subset T\setminus \omega(X_L)$ as otherwise it cannot be contained in $T$ by the definition of oscillatory quasi-arc. Then it implies that $U\subset \Omega\setminus \omega(X_L)$, which by the fact that singleton fibers of $\eta$ are dense in $T$ implies that $U$ must be contained in fiber of $\eta$. But then it is nowhere dense in $T$ which again is a contradiction, therefore the claim holds.
It shows that $\phi_L$
satisfies conditions from Definition~\ref{def:tranched_graph}, so we get that $X_L$ is a hereditary tranched graph.
\end{proof}
\begin{figure}[ht]

            \centering
            \includegraphics[scale=0.18]{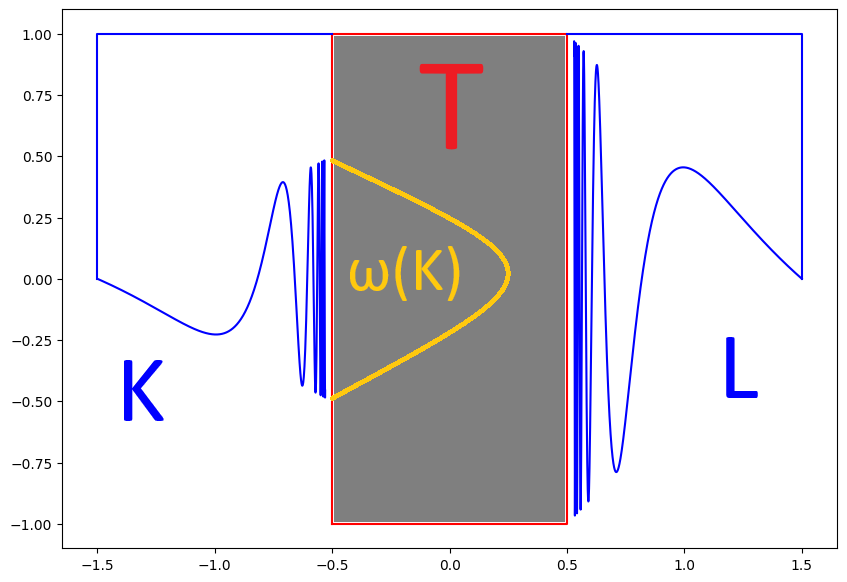}
            \includegraphics[scale=0.18]{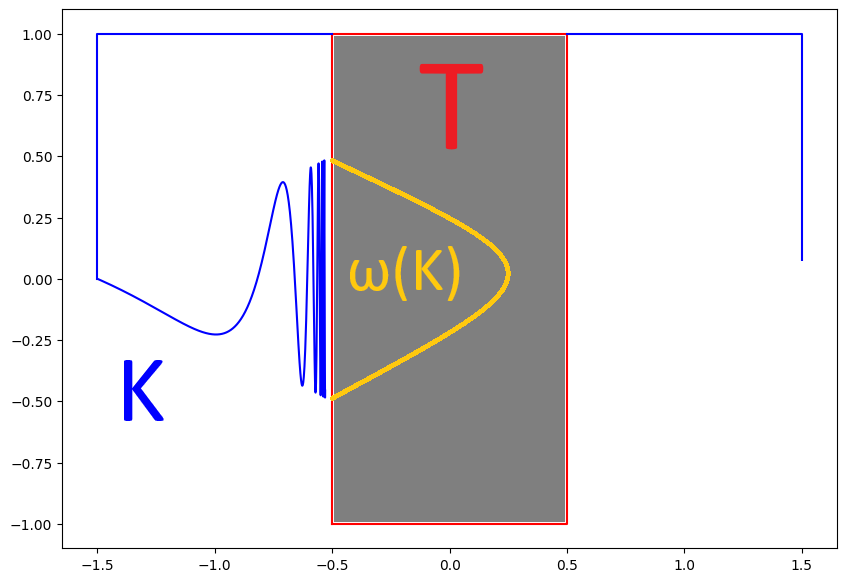}
            \includegraphics[scale=0.18]{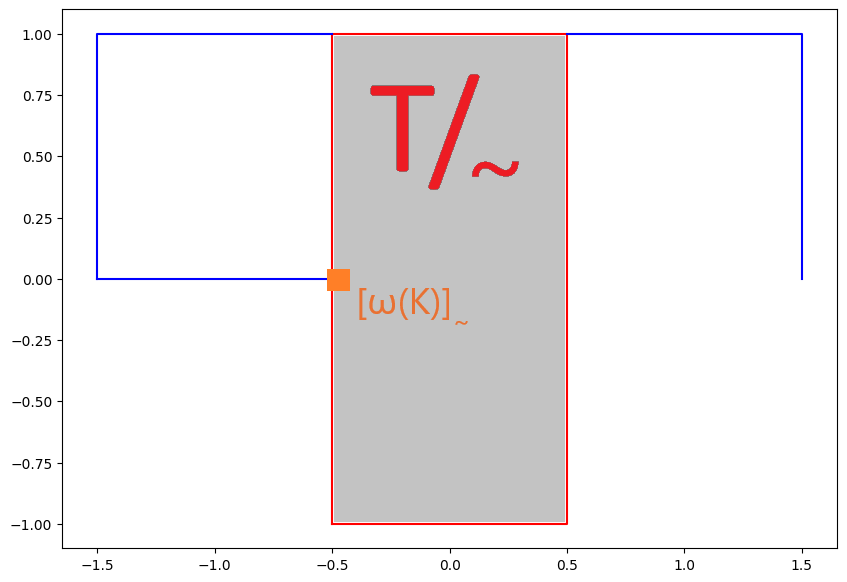}
            \includegraphics[scale=0.18]{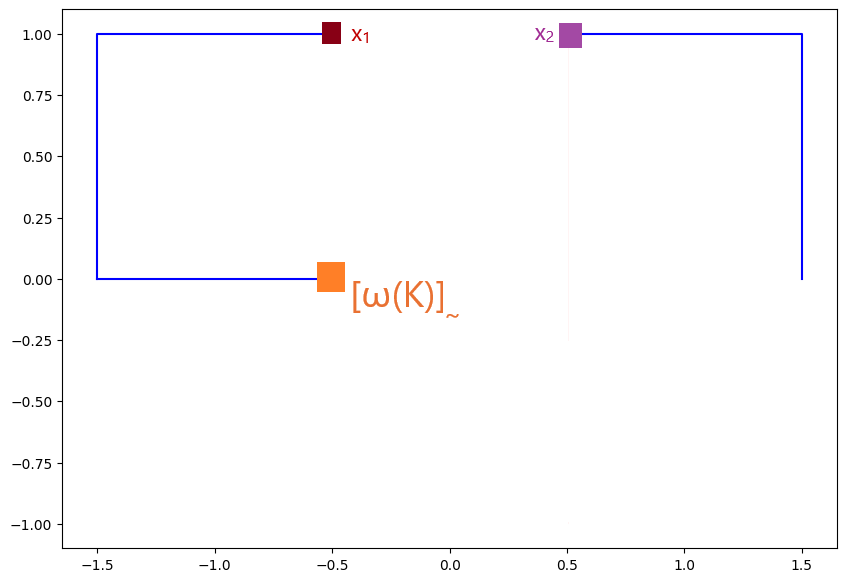}
            \includegraphics[scale=0.18]{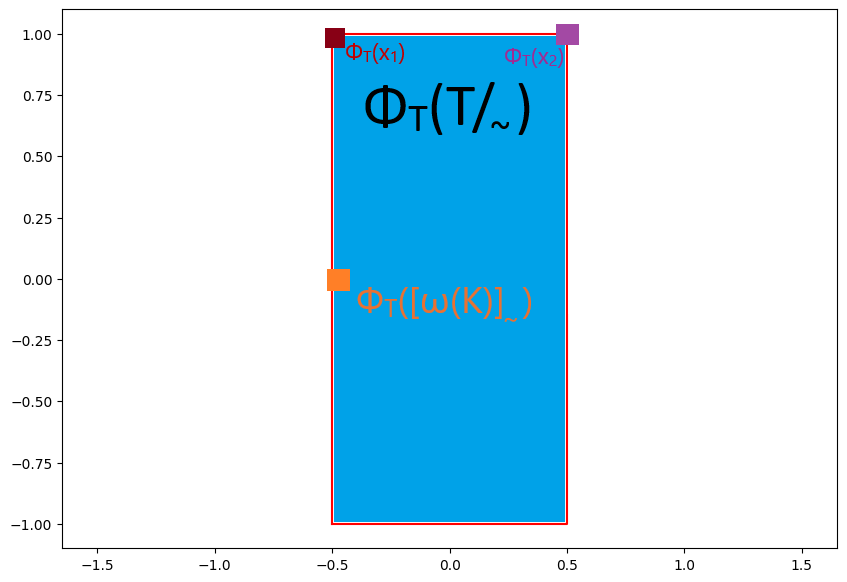}
            \includegraphics[scale=0.18]{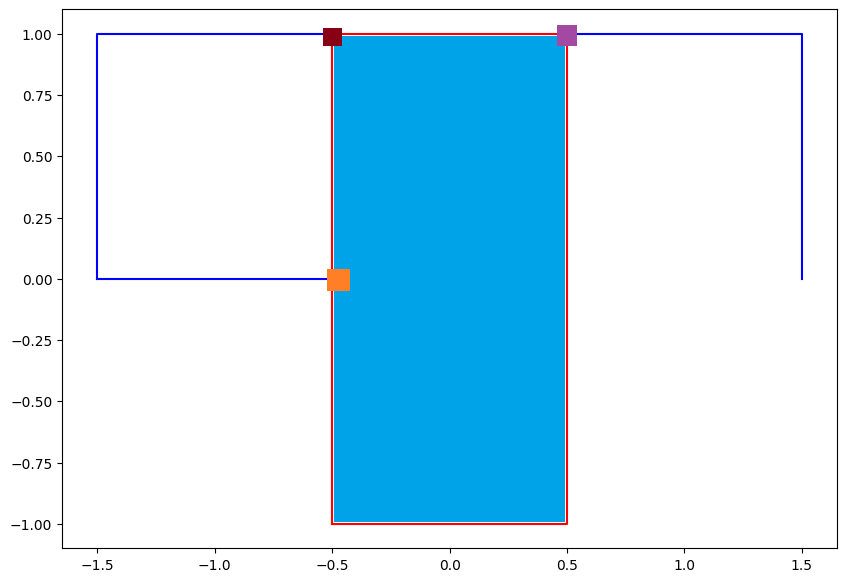}
            \caption{A sketch of constructions in Lemma~\ref{lem:regular_finite_delete_quasi_arc}, from top-left:
Continuum $X$, continuum $X_L$ and continuum $X_L/_\sim$. On the bottom $X_1$ and $\phi_T(T/_\sim)$ with points that will be identified marked in the same color.}
            \label{fig:removing_arcs}
    \end{figure}

In the next lemma we prove that for arcwise connected continua, the image of a tranche has to be contained in a circle. This is in line with the arguments shown before in the paper.
\begin{lem}
\label{lem:tranche_in_circle}
    Let $X$ be an arcwise connected tranched graph with an associated mapping $\phi \colon X \ra Y$ from the definition. Suppose that $y \in Y$ defines a nondegenerate fiber $\phi^{-1}(y)$. Then there is a circle $S \subset Y$ such that $y \in S$.
\end{lem}
\begin{proof}
 Suppose that $y \in Y$ is such that $T=\phi^{-1}(y) \subset X$ is a tranche and there is no circle $S \subset Y$ such that $y \in S$. It is easy to see that $y$ cannot be an endpoint of $Y$. It follows that $Y \backslash \{y\}$ has at least two connected components.

By Lemma~\ref{lem:finite_tranches} the tranche $T$ is a union of limit sets of oscillatory quasi-arcs. Choose an oscillatory quasi-arc  without ancestors $L$  such that $\omega(L) \subset T$ and let $Y_2$ be a connected component of $Y \backslash \{y\}$ that does not intersect with $\phi(L)$. Choose $x_1 \in L$ not contained in a tranche and denote $y_1=\phi(x_1)$. Similarly, there is $y_2 \in Y_2$ such that $\phi^{-1}(y_2)$ is a singleton, fix the unique $x_2\in X$ with $y_2=\phi(x_2)$.
    
The point $y$ is not a subset of a circle, hence any arc with endpoints $\{y_1,y_2\}$ needs to contain $y$ as its element. Take any arc $A\subset X$ with the endpoints $x_1$ and $x_2$, which exists since $X$ is arcwise connected.  It follows that $y \in \phi(A)$. 

This means that $A \cap (L \cup \omega(L))$ is an arc because otherwise $Y$ contains a circle intersecting $y$. This, combined with the fact that $\omega(L) \cap A \neq \emptyset$ gives us that $L$ is not oscillatory, which is a contradiction. The proof is finished.\end{proof}

\begin{defn}
    We say that a continuum $X$ is irreducible between $a$ and $b$ for points $a,b \in X$ if the only subcontinuum containing both $a$ and $b$ is $X$.

    A continuum $X$ is irreducible between $x_0$ and $x_1$ for points $x_0,x_1 \in X$ is called a $\lambda$-continuum if there is a monotone map $\phi \colon X \to [0,1]$ such that $\phi^{-1}(i)=x_i$ for $i \in \{0,1\}$ and every fiber $\phi^{-1}(y)$ is connected. 
\end{defn}

 Suppose $X$ is a $\lambda$-continuum. If the union of degenerate fibers is dense in $X$, then it is a tranched graph. This immediately gives the following.  

\begin{cor}
    Suppose $X$ is both an arcwise connected tranched graph and a $\lambda$-continuum. Then $X$ is an arc.
\end{cor}
Using Lemma~\ref{lem:tranche_in_circle}  we easily obtain the following. 
\begin{cor}
     Suppose $X$ is an arcwise connected tranched graph and let $\phi \colon X \ra Y$ be a mapping from the definition. If $a \in \End(Y)$ and $b \in \Br(Y)$ are such that $[a,b] \cap \Br(Y) = \{b\}$, then for all $y \in [a,b)$ the fiber $\phi^{-1}(y)$ is degenerate.

\end{cor}
 The following can be used to find an upper bound for the number of tranches.
\begin{lem}
\label{lem:upper_bound_tranches}
    Suppose $X$ is an arcwise connected tranched graph, let $\phi \colon X \ra Y$ be the mapping from the definition and assume that $y_1 \neq y_2 \in Y$ define tranches. Then there are circles $S_1 \neq S_2 \subset Y$ such that $y_i \in S_i$.
\end{lem}
\begin{proof}
     Assume on the contrary that there are no circles $S_1 \neq S_2 \subset Y$ such that $y_i \in S_i$.
    By Lemma~\ref{lem:tranche_in_circle} there is a unique circle $S$ such that $y_1,y_2 \in S$. We claim that $S$ needs to have at least two branching points of $X$.  
    Suppose it is not the case and denote by $C_1, C_2$ connected components of $S \backslash \{y_1,y_2\}$. Then, if we choose any points $c_i \in C_i$ for $i=1,2$, then any arc connecting them has to pass through the point $y_1$ or $y_2$. It is a contradiction (cf. the proof of Lemma~\ref{lem:tranche_in_circle}), so the claim holds.

    Fix two distinct $b_1,b_2 \subset S \cap \Br(Y)$ and let $C_1,C_2$ be connected components of $S \backslash \{b_1,b_2\}$ such that $y_i \in C_i$. The continuum $X$ is arcwise connected, for points $c_i \in C_i$ there is an arc $A\subset X$ with the endpoints in the sets  $\phi^{-1}(c_1)$ and $\phi^{-1}(c_2)$. 
    Using the same argument as before $y_1,y_2\not\in\phi(A)$.
    But then $S_i=\phi(A) \cup C_i$ are circles such that $y_i \in S_i$ and $S_1 \neq S_2$, which leads us to a contradiction with the uniqueness of $S$ and completes the proof.
\end{proof}

Using the standard terminology of algebraic topology, we can state the result as follows: 
Using Lemma~\ref{lem:upper_bound_tranches} inductively, we get that for a  tranched graph $X$, with associated mapping $\phi \colon X \ra Y$ we have to have at least as many circles in the topological graph $Y$ as there are tranches in $X$.

\begin{prop}
\label{prop:tranches_betti}
    Suppose $X$ is an arcwise connected tranched graph and let $\phi \colon X \ra Y$ be the mapping from its definition. Then $X$ has at most $b_1(Y)$ tranches, where $b_1(Y)$ is the first Betti number of the topological graph $Y$.
\end{prop}
We can use Proposition~\ref{prop:tranches_betti}
inductively to prove a finite number of tranches at any depth of the continuum.
Strictly speaking, we have the following.
\begin{lem}
\label{lem:arcwise_width}
    Let $X$ be an arcwise connected, hereditary tranched graph. Then for all $ k\in \N$, the set $lvl_k(X)$ is finite.
\end{lem}
\begin{proof}
    The Betti number of a graph is always finite; hence, by Proposition~\ref{prop:tranches_betti} the set $lvl_1(X)$ is finite. By Lemma~\ref{lem:finite_tranches}, the set $lvl_1(X)$ 
    can be presented as the union of limit sets of (finitely many) oscillatory quasi-arcs in $X$. Using Lemma~\ref{lem:regular_finite_delete_quasi_arc} we remove these quasi-arcs from $X$, obtaining an arcwise connected continuum $X_1$ such that $lvl_1(X_1)=lvl_2(X)$. Repeating the above argument, we obtain that the set of tranches of $X_1$ is finite and hence so is the set $lvl_2(X)$.  Continuing this process inductively, we see that for every $N \in \N$ set $lvl_N(X)$ is finite, completing the proof. 
\end{proof}
We can now prove the theorem characterizing tranched graphs which at the same time satisfy the definition of quasi-graph.
\begin{thm}
\label{thm:what_tranched_are_quasi}
    Let $X$ be a tranched graph. Then $X$ is a quasi-graph if and only if it is an arcwise connected, hereditary tranched graph of finite depth.
\end{thm}
\begin{proof}
    Let $X$ be a tranched graph.
   Assume first that $X$ is a quasi-graph. We immediately get that $X$ is arcwise connected. As we showed in  Lemma~\ref{limit_collapse} and Theorem~\ref{thm:commonbase}, in this case all tranches are the limit sets of oscillatory quasi-arcs. By Definition~\ref{def_quasigraph} there are finitely many of them. It means that $X$ has a finite set of tranches and none of them contains an infinite hierarchy of tranches. Finally, Lemma~\ref{lem:quasi-graphs_are_tranched} shows that $X$ is a hereditary tranched graph.
    
    Now suppose   $X$  is a hereditary tranched graph which additionally is arcwise connected and of finite depth. Let  $\phi\colon X \ra Y$ be the map from Definition~\ref{def:tranched_graph}. By Lemma~\ref{lem:finite_tranches} all tranches are connected unions of limit sets of quasi-arcs. We use Lemma~\ref{lem:regular_finite_delete_quasi_arc} inductively to remove all oscillatory quasi-arcs from $X$, until we get a topological graph, which we denote by $G$. Suppose we removed $n$ oscillatory quasi-arcs from $X$ to get $G$. Let us assume that the indexes of oscillatory quasi-arcs we removed are ordered in a way that $L_n$ is the first quasi-arc we removed and  $L_1$ the last.  Then we get that:
    \begin{enumerate}[(i)]
        \item  $X =G \cup \bigcup_{j=1}^n L_j$ and by our assumption that quasi-arcs have no branching points $\Br(X) \subset G$ and $\End(X) \subset G$.
        \item  For any oscillatory quasi-arc $L_j$ only intersection with topological graph $G$ is at its endpoint.
        \item  Using the order we indexed the quasi-arcs, for any  $i=1, \ldots,  n$  quasi-arcs with lower indexed $j<i$ were removed from $X$ later than $L_i$ and those with higher index $j>i$ were removed before $L_i$. This means that $\omega (L_i) \subset G \cup  \bigcup_{j=1}^{i-1} L_j$ for any $ 1 \leq i \leq n$.
        \item Suppose now that $\omega (L_i) \cap L_j \neq \emptyset$ for some indexes  $i,j \in \N$. Then  $j<i$, and so
         $L_i$ will be removed  in the construction leading to $G$  before $L_j$ is removed.
        If $L_j$ was not a subset of $\omega(L_i)$, we can shorten $L_j$ accordingly obtaining that  $\omega (L_i) \cap L_j = \emptyset$, so we may assume that  $\omega (L_i) \subset L_j$ provided that $\omega (L_i) \cap L_j \neq \emptyset$ for some $i,j$.
    \end{enumerate}
   We obtained that $X$ satisfies all the properties from Definition~\ref{def_quasigraph}, meaning it is a quasi-graph, ending the proof.
 \end{proof}
 
The following example shows that tranches of a generalized sin(1/x)-continuum are not necessarily generalized sin(1/x)-type continua themselves.
\begin{exmp}
\label{exmp:regular_no_sin}
  Let
    \begin{align*}
		G &= (\{-1\} \times [-\frac{1}{2},1]) \cup ([-1,0]\times \{1\})  \cup (\{0\} \times [-1,1])\\
		L_1 &= \{(x,\frac{1}{2} \sin(-\pi/x)(1+x) +1 )~|~ x \in [-1,0)\} \\
		L_2 &=  \{(x,\frac{1}{2} \sin(-\pi/x)(1+x) -1 )~|~ x \in [-1,0)\} \\
		X_1 &= G \cup L_1 \cup L_2
    \end{align*}
    Then, as it does not satisfy the necessary condition in Theorem~\ref{thm:commonbase}, it is not a generalized sin(1/x)-type continuum (see Figure~\ref{fig:dopelnienie}). 
    \begin{figure}
            \centering
            \includegraphics[scale=0.3]{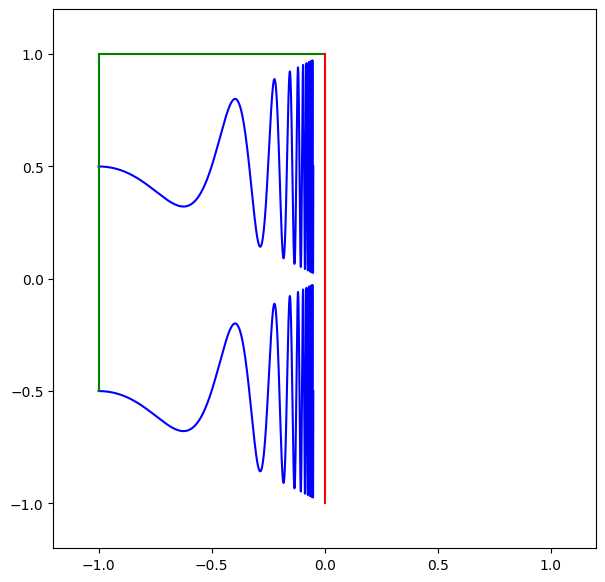}
            \includegraphics[scale=0.22]{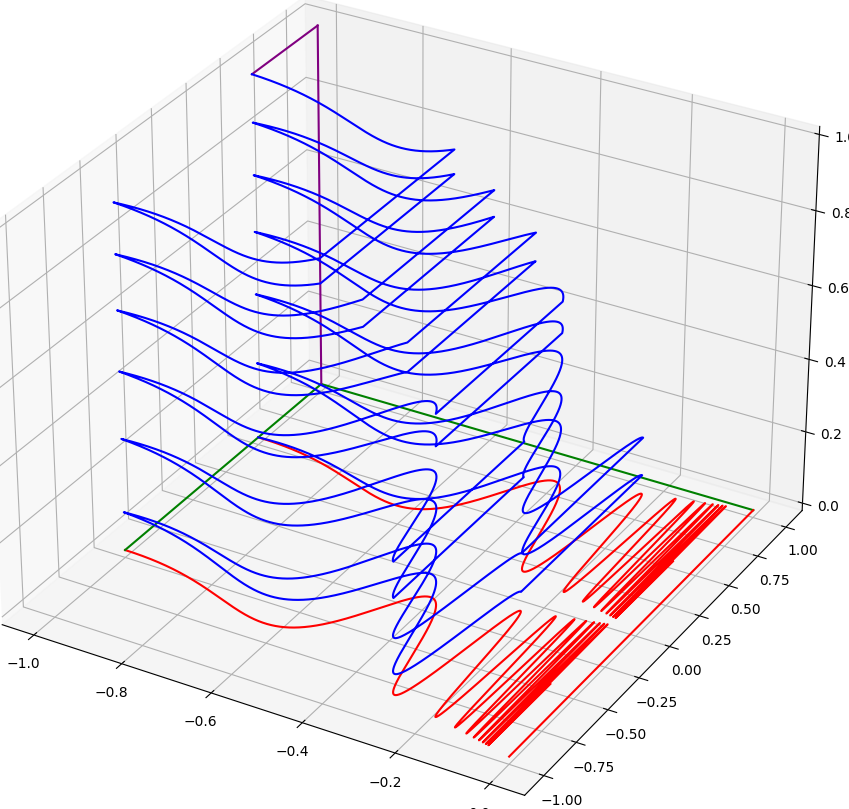}
           \caption{The quasi-graphs $X_1$ and $X$ from Example~\ref{exmp:regular_no_sin}} 
            \label{fig:dopelnienie}
        \end{figure}

Denote by $\varphi_1,\varphi_2$ parametrizations of quasi-arcs $L_1$ and $L_2$ respectively.
    Let  $\gamma_N^i \colon [0,1] \ra X_1$   for $N \in \N$ be curves defined as:
\begin{center}

\begin{tabular}{ c c c }
 $\gamma_N^1(t)=\varphi_1(Nt)$, & $\gamma_N^2(t)=(1-t)\varphi_1(N)+t\varphi_2(N)$, & $\gamma_N^3(t)=\varphi_2(N-Nt)$, \\ 
 $\gamma_N^4(t)=\varphi_2(Nt)$, & $\gamma_N^5(t)=t\varphi_1(N)+(1-t)\varphi_2(N)$, & $\gamma_N^6(t)=\varphi_1(N-Nt)$. 
\end{tabular}
\end{center}
The parameter $N$ decides how far into the quasi-arc $L_1$ we get, before we move onto $L_2$. Let $\gamma_N$ be such that $\gamma_{N}(t)=\gamma_N^i(6t)$  for $t \in [\frac{i-1}{6},\frac{i}{6}],~i=1,2,3,4,5,6$ and let  $\gamma$ be a curve defined as $$ \gamma(Nt)=\gamma_N(t).$$
    Define a continuous map $\varphi_K \colon [0,\infty) \ra X_1 \times \R$  acting in the following way:
   $$ \varphi_K(t)=(\gamma(t), \frac{1}{t+1})$$

    and denote by $A$ an arc connecting set $X_1 \times \{0\}$ with point $\varphi_K(0)$. If we write $K=\varphi_K([0,\infty))$, then
    $$X=(X_1 \times\{0\}) \cup K \cup A. $$
    By Definition~\ref{def_quasigraph} the continuum $X$ is a quasi-graph  (see Figure~\ref{fig:dopelnienie}). By Lemma~\ref{lem:quasi-graphs_are_tranched} we get that both $X$ and its only tranche $X_1$ are tranched graphs. However, $X_1$ is not a generalized sin(1/x)-type continuum. 

    Therefore, we constructed a generalized sin(1/x)-type continuum, that is also a quasi-graph, whose tranche is not a generalized sin(1/x)-type continuum.
  
\end{exmp}
All generalized sin(1/x)-type continua are tranched graphs by the definition, hence we get the following characterization as a direct consequence of Theorem~\ref{thm:what_tranched_are_quasi}.
 \begin{cor}
      Let $X$ be a generalized sin(1/x)-type continuum. Then $X$ is a quasi-graph if and only if it is arcwise connected,   hereditary tranched graph, and of finite depth.
 \end{cor}
\section{Construction of an arcwise connected generalized sin(1/x)-type continuum   that is a tranched graph of infinite depth}\label{sec:example}

 The aim of this section is to rigorously  prove the correctness of the construction advertised in its title (see also Example~\ref{exmp:infinite_depth}). The section can be read independently, so the reader may freely skip the
  details and continue reading in Section~\ref{sec:dynamics}.

 In Example~\ref{exmp:infinite_order_quasi_arc} we constructed an infinite depth hereditary tranched graph, hence the assumption that continuum is of finite depth is necessary in Theorem~\ref{thm:what_tranched_are_quasi}. By the method of construction, this continuum was not arcwise connected. On the other hand, we proved in Lemma~\ref{lem:arcwise_width}
    that in arcwise connected hereditary tranched graph on any level (depth) the set of tranches is finite.
  Next example shows that 
  in these continua still
  infinite depth is possible (the example is even a generalized sin(1/x)-type continuum).
 In the proof we will use yet another geometrical representation of sin(1/x)-type curve, showing how changing the geometry of the space allows us for different constructions, depending on our goals.

%  Before going into technical details, 
  Let us describe informally the construction we are going to perform (see Figure~\ref{fig:warsaw_infinite} for a sketch of the first step).   
  We start with the Warsaw circle $X_0$ in the plane and we consider quasi-arc $L_0\subset X_0$ whose limit is an interval $ \omega(L_0)$. 
 Then we modify $L_0$ to $L_1$ by incorporating an oscillation in an additional dimension in such a way that now $\omega(L_1)\supset \omega(L_0)$
 is a smaller copy of $X_0$ perpendicular to the original $X_0$. In this way a continuum $X_1$
 is obtained. This process is repeated inductively, where each time $\omega(L_n)$
 is replaced by a smaller copy of $X_n$ and $L_n$ is modified to $L_{n+1}$ which oscillates in one direction more than $L_n$.
 As the ultimate step, we prove that the limit continuum $X_\infty$ of the sequence $X_n$ in the Hilbert cube is in fact an arcwise connected generalized sin(1/x)-type continuum.
 The self-similarity imposed by the construction is then used to show that $X_\infty$ has infinite depth.

     \begin{exmp}
\label{exmp:infinite_depth}
    There exists an arcwise connected generalized sin(1/x)-type continuum   that is a tranched graph of infinite depth. It will be constructed by induction.
\end{exmp}
    Let  us start with the function
    $$
    f(t) =\begin{cases} \frac14( \sin\frac{\pi}t+1+3t) & \text{if } t \in (0,\frac12],\\ \frac54-\frac54t& \text{if } t  \in [\frac12, 1] \end{cases}
    $$ Let $\varphi(t)=(t,f(t))$  and $X=([0,1]\times\{0\}) \cup \overline{L}$ be the Warsaw circle, where  $L=\varphi((0,1]))$, i.e. we choose the geometric representation as in Figure~\ref{fig:warsaw_infinite}.

    \begin{figure}[ht]

            \centering
             \begin{subfigure}[b]{0.3\textwidth}
      
        \includegraphics[scale=0.25]{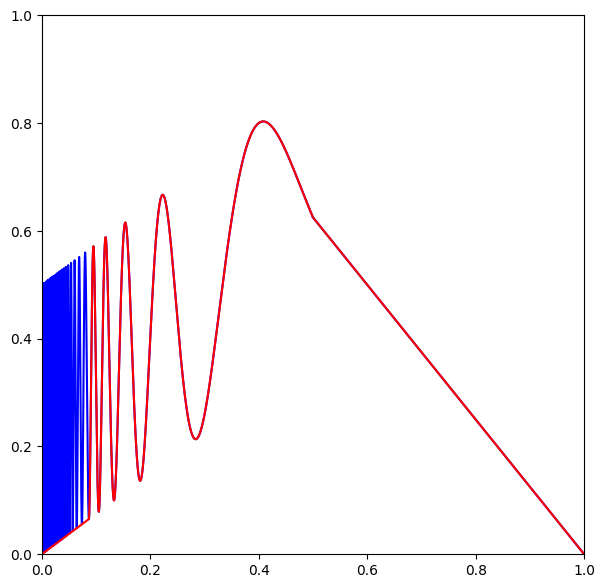}
        \caption{}\label{fig11:A}
            \end{subfigure}
            \begin{subfigure}[b]{0.3\textwidth}

        \includegraphics[scale=0.25]{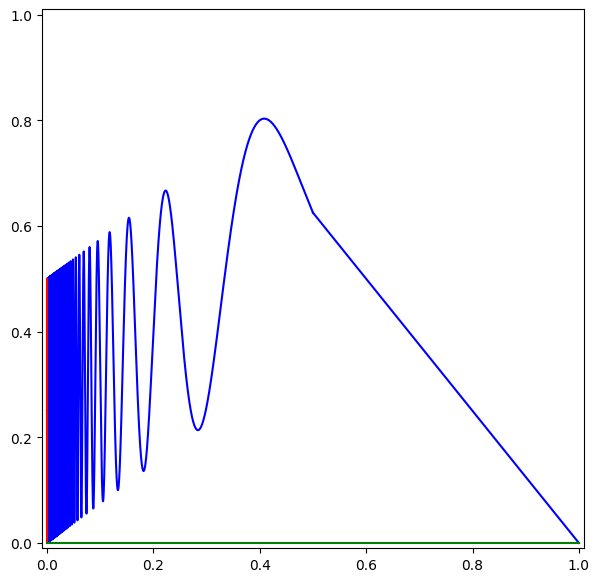}
        \caption{}\label{fig11:B}
            \end{subfigure}
            \begin{subfigure}[b]{0.3\textwidth}

         \includegraphics[scale=0.15]{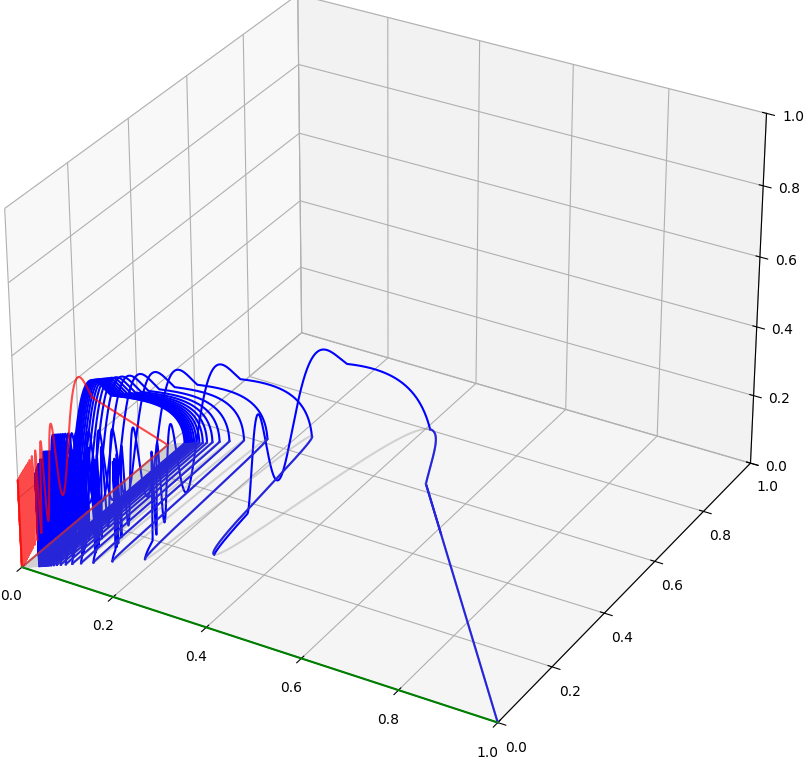}
        \caption{}\label{fig11:C}
            \end{subfigure}

            \caption{From left to right: (A) graphs of $f_{5}$ (in red) and $f$
            (in blue, beyond graph overlap with $f_5$), (B) the Warsaw circle $X$ and (C) the continuum $X_1$ from Example~\ref{exmp:infinite_depth}}.

            \label{fig:warsaw_infinite}
    \end{figure}

    Denote by $\{z^{(i)}\}_{i \in \N}$, $\{y^{(i)}\}_{i \in \N}$, the sets of local maxima and minima of $f$, with ordering $z^{(m)}<z^{(n)}$ if and only if $m>n$. By definition  $y^{(1)}=1$, hence $y^{(n+1)}<z^{(n)}<y^{(n)}$
for every $n\in \N$.
    For $N>1$ let 
    $$
    f_{N-1}(t) = 
    \begin{cases} 
    f(t) & \text{if } t \in [y^{(N)},1],\\
    \frac{f(y^{(N)})}{y^{(N)}}t & \text{if }  t \in [0, y^{(N)}] 
    \end{cases},
    $$

   \begin{enumerate}
        
        \item $\varphi_0 (t)=(t,f(t),0,\ldots)$, $L_0=\{\varphi_0(t), t \in (0,1]\} \subset \mathcal{H}$
        \label{cond:star:1}
        \item  $G=[0,1] \times \{0\}^\infty$ and $X_0=G \cup \overline{L_0}$
        \label{cond:star:2}
        \item $P^0 = \{P^0_i\}$, where $P^0_i=[y^{(i+1)},z^{(i)}]$. \label{cond:star:3}
    \end{enumerate}
    Note that \begin{enumerate}\setcounter{enumi}{3}
        \item \label{cond:star:4} $f_i|_{P^0_j}=f|_{P^0_j}$ for $j \le i$.
   \end{enumerate}

    For $i_0 \in \mathbb{N}$ let $h_{i_0}$ be an order-preserving affine homeomorphism from $f(P^0_{i_0})$ to $[0,1]$. 
    \begin{lem}\label{lem:Pn:construct}
There is a collection  of closed intervals $P^n=\{P^n_{i_0,\ldots i_n}\}_{i_0\in \mathbb{N}, i_0\geq i_1 \cdots\ge i_n}$ 
for $n=0,1,\ldots$ such that for any $n>0$ and admissible sequence of indexes we have $P^n_{i_0,\ldots,i_n} \subset P^{n-1}_{i_0,\ldots,i_{n-1}} $ and the equality $(h_{i_0} \circ f)(P^n_{i_0,\ldots,i_n})=P^{n-1}_{i_1,\ldots,i_n} $ holds.
    \end{lem}
    
    \begin{proof}
        We will proceed by induction. As was defined before, $P^0_{i_0}=[y^{(i_0+1)},z^{(i_0)}]$ and $P^0=\{P^0_{i_0}\}_{i_0 \in \mathbb{N}}$. Take any $i_0 \in \mathbb{N}$ and $i_1 \le i_0$. By the definition $h_{i_0}(f(P^0_{i_0}))=[0,1]$, hence $P^0_{i_1} \subset (h_{i_0} \circ f)(P^0_{i_0})$, so there is a subinterval of $P^0_{i_0}$ whose image under $(h_{i_0} \circ f)$ is   $P^0_{i_1}$. We define  $P^1_{i_0,i_1}$ to be this  nonempty closed interval.

       Suppose we have constructed $P^k$ for all $k \leq n$. Choose any natural numbers $i_0 \geq i_1 \geq \cdots\geq i_{n+1}$. By definition $(h_{i_0} \circ f)(P^{n}_{i_0,\ldots,i_{n}})=P^{n-1}_{i_1,\ldots,i_n}$ and there is a closed  subinterval $P^n_{i_1,\ldots,i_n,i_{n+1}} \subset P^{n-1}_{i_1,\ldots,i_n}$, that is $P^n_{i_1,\ldots,i_{n+1}} \subset (h_{i_0} \circ f)(P^{n}_{i_0,\ldots,i_{n}})$. Therefore, there is a closed interval $A \subset P^{n}_{i_0,\ldots,i_{n}}$ such that $P^n_{i_1,\ldots,i_{n+1}} = (h_{i_0} \circ f)(A)$. We put $P^{n+1}_{i_0,\ldots i_{n+1}}=A$, 
       completing the induction.      
    \end{proof}
    
    In what follows, we assume that the collections $P^n$ of closed intervals (where $n=0,1,2,\ldots$) provided by Lemma~\ref{lem:Pn:construct} are fixed.
    \begin{lem}
        Let $P^n_{i_0,\ldots,i_n} \in P^n$. Then the map 
        $$
        g_{i_0,\ldots i_n} \colon P^n_{i_0,\ldots,i_n} \ni t \mapsto (f_{i_n} \circ h_{i_n} \circ \cdots  \circ h_{i_1} \circ f_{i_0} \circ h_{i_0} \circ f )(t) \in [0,1]
        $$ 
        is well-defined and continuous.
    \end{lem}
    \begin{proof}
        Fix any $t \in P^n_{i_0,\ldots,i_n}$, we will go from the inside out to prove our claim.
Note that
            $f(t) \in f(P^n_{i_0,\ldots,i_n})$,
            so $(h_{i_0} \circ f)(t) \in P^{n-1}_{i_1,\ldots,i_n} \subset P^0_{i_1}$ and $i_1 \leq i_0$.
            Therefore $f_{i_0}(s)=f(s)$ for $s=(h_{i_0} \circ f)(t)$ by \eqref{cond:star:4}. Suppose we know that $(h_{i_k} \circ  f_{i_{k-1}} \circ  \cdots \circ h_{i_1} \circ f_{i_0} \circ h_{i_0} \circ f)(t) \in P^{n-k}_{i_{n-k},\ldots,i_n}$ for some $k<n$, denote $s=(h_{i_k} \circ \cdots  \circ f)(t)$. As $i_{k+1} \leq i_k$, by \eqref{cond:star:4} we have that $f_{i_k}(s)=f(s)$, so
     $f_{i_k}(s)=f(s) \in f(P^{n-k}_{i_{n-k},\ldots,i_n})=h^{-1}_{i_{k+1}}(P^{n-k-1}_{i_{n-k-1},\ldots,i_n}) $. Then
        $(h_{i_{k+1}} \circ f_{i_k} \circ h_{i_k} \circ \cdots  \circ f)(t) \in  P^{n-k-1}_{i_{n-k-1},\ldots,i_n}$.    
        Completing the induction, $(h_{i_{n}} \circ f_{i_{n-1}} \circ h_{i_{n-1}} \circ \cdots  \circ f)(t) \in  P^{0}_{i_n} $, so the map $f_{i_n}$ is well defined for this point and consequently $g_{i_0,\ldots,i_n}$ is well defined. The map $g_{i_0,\ldots,i_n}$ is continuous as a composition of continuous maps. This finishes the proof.
    \end{proof}
    \begin{lem}
    \label{lem:glueing}
            For any $i_0 \ge i_1 \geq \cdots \geq i_n$ we have $g_{i_0, \ldots, i_n} (\End (P^n_{i_0, \ldots, i_n}))=\{0\}$.
    \end{lem}
   
    \begin{proof}
     Fix any $i_0 \ge i_1 \geq \cdots \geq i_n$.
The map $h_{i_0}$ is a homeomorphism from $f(P^0_{i_0})$ to $[0,1]$,  and $f_{i_0}(\{0,1\})=\{0\}$, so we have that $g_{i_0}$ sends endpoints of $P^0_{i_0}$ to $\{0\}$ for any $i_0 \in \mathbb{N}$. Suppose the result holds for any $k<n$. Choose any endpoint $e \in \End (P^n_{i_0,\ldots,i_n})$. By definition $(h_{i_0} \circ f)(P^n_{i_0,\ldots,i_n})=
P^{n-1}_{i_1,\ldots,i_n}$ 
 and $(h_{i_0} \circ f)$ is a homeomorphism on $P^0_{i_0}$, hence we have that  $(h_{i_0} \circ f)(e)=e' \in \End (  P^{n-1}_{i_1,\ldots,i_{n}} )$. As $e' \in P^{n-1}_{i_1,\ldots,i_{n}}  \subset P^0_{i_1}$ and $i_1 <i_0$ we have that $f(e')=f_{i_0}(e')$ which implies that $g_{i_0,\ldots i_n}(e)= (f_{i_n} \circ h_{i_n} \circ \cdots  \circ h_{i_1} \circ f_{i_0} \circ h_{i_0} \circ f )(e)=(f_{i_n} \circ h_{i_n} \circ \cdots  \circ h_{i_1} \circ f_{i_0} )(e')=(f_{i_n} \circ h_{i_n} \circ \cdots  \circ h_{i_1} \circ f )(e')=g_{i_1,\ldots,i_n}(e')=0$.
    \end{proof}

    We will use a standard notation for $x \in \mathcal{H}$ denoting $x=(x_0,x_1,\ldots)$. Let  $\pi_n \colon \mathcal{H} \ni x\mapsto x_n \in [0,1]$ be the projection onto the $(n+1)th$ coordinate,

    In the induction, we use the family $P^n$, but it is fixed already, before we start the recursive construction. Inductively we will construct $(X_n,L_n,\varphi_n)$ with the following properties:
    \begin{enumerate}
        \item[(A1)]\label{con:A1} $X_n=G \cup \overline{L_n} \subset \mathcal{H}$ and $L_n \subset (0,1]^{n+1} \times \{0\}^\infty$,
        \item[(A2)] $L_n$ is an oscillatory quasi-arc in $X_n$ with parametrization $\varphi_n \colon (0,1] \to L_n$,
        \item[(A3)] $X_n /_\sim$ is a circle,
        \item[(A4)] the unique nondegenerate fiber of $\pi_\sim|_{X_n}$ is $\omega(L_n)$
        \item[(A5)] 
        the map $\pi_0|_{L_n} \colon L_n \to (0,1]$ is a homeomorphism.\label{con:A5}
        \item[(A6)] for every subcontinuum $Y \subset \omega(L_n)$ and every $\epsilon>0$ there is an arc $[a,b] \in (0,1]$ such that $d_H(Y,\varphi_n([a,b]))<\epsilon$.

        \item[(A7)]for $n>0$ and for all $x \in L_n$ we have $(x_0,\ldots,x_n,0,\ldots) \in L_{n-1}$,\label{con:A7}
        \item[(A8)]for $n>0$ we have $\omega(L_n)=\frac12\theta(X_{n-1})$, where $\frac12\theta((x_0,x_1,\ldots))=(0,\frac12 x_0,\frac12 x_1, \ldots)$
        \item[(A9)] for $n>0$, any  natural numbers $  i_1 \geq \cdots\ge i_n$ and any $y\in \frac12\theta(\varphi_{n-1}(P^{n-1}_{i_1,\ldots i_{n}}))$
        there is a sequence $t^{(i_0)}\in P^n_{i_0,\ldots i_{n}}$, $i_0\geq i_1$ such that
        $\lim\limits_{i_0 \to \infty}\varphi_n(t^{(i_0)}) =y$.
      
        \item[(A10)] for $n>0$, for any $t \not \in \bigcup_{P \in P^n}P$ we have $\pi_{n+1}(\varphi_n(t))=0$ and for any $\{{i_1,\ldots i_n}\}$ with $  i_0 \geq \cdots\ge i_n$ and any $t \in P^n_{i_0,\ldots, i_n}$ we have $\pi_{n+1}(\varphi_{n}(t))=\frac1{2^{n}}g_{i_0, \ldots, i_{n-1}}$\label{con:A10}

    \end{enumerate}

    Also, from $(A5)$ and $(A7)$ we get that $\pi_k|_{L_n} \equiv 0$ for every $k> n+1$. 
    
    \begin{lem}
        The triple $(X_0,L_0,\varphi_0)$ satisfies the conditions $(A1)-(A10)$.
    \end{lem}
    \begin{proof}
        Since $n=0$, we only have to check $(A1)-(A6)$. The continuum $X_0$ is a Warsaw circle defined by \eqref{cond:star:2}, hence $(A1)-(A4)$ and $(A6)$ hold.  Note that any $x \in L_0$ with $x_0=t$  is uniquely defined by $(t,f(t),0,\ldots)$, which gives us $(A5)$.
    \end{proof}

    For the general step of the induction, suppose we have already defined $(X_n,L_n,\varphi_n)$ which satisfy $(A1)-(A10)$. We will construct $(X_{n+1},L_{n+1},\varphi_{n+1})$ which satisfy $(A1)-(A10)$ as well.

    First, for all $t \in (0,1]\setminus \bigcup_{P \in P^n}P$ we define $\varphi_{n+1}(t)=\varphi_n(t)$. Fix any  $P =P^n_{i_0,\ldots,i_n}$ and $t \in P $, and put:
    $$ \varphi_{n+1}(t)=(t=\pi_0(\varphi_n(t)),\ldots,\pi_{n-1}(\varphi_n(t)),\pi_n(\varphi_n(t)), \frac1{2^{n}}g_{i_0, \ldots, i_n}(t),0,0,\ldots)$$
    
    By Lemma~\ref{lem:glueing} we get that the map $\varphi_{n+1}$ is continuous. We denote $L_{n+1}=\varphi_{n+1}((0,1])$ and  put $X_n = G \cup \overline{L_{n+1}}$. This way the triple $(X_{n+1},L_{n+1},\varphi_{n+1})$ is defined.

    We see that $(A1)$ holds just by the definition. By definition $\varphi_{n+1}$ is injective, hence we get $(A2)$. Every $x \in L_{n+1}$ is uniquely determined by $x=\varphi_{n+1}(x_0)$, so in particular $(A5)$ holds. 
    
    Choose any $t\in (0,1]$ and let $x=\varphi_{n+1}(t) \in L_{n+1}$. Then $(x_0,\ldots,x_{n}, 0,\ldots)=(\pi_0(\varphi_n(t)),\ldots,\pi_{n-1}(\varphi_n(t)),\pi_n(\varphi_n(t)),0,\ldots) \in L_n$, proving $(A7)$.  
    Directly form the definition of $(n+2)$nd coordinate of $\varphi_{n+1}$, we get $(A10)$.

    As $L_{n+1}$ is oscillatory and $x_0=0$ for all $x \in \omega(L_{n+1}) $, it is a nondegenerate fiber of $X_{n+1}$. Furthermore, note that  by definition, $x_0>0$ for all points $x \in L_{n+1} $. By $(A1)$ for $n+1$ proven above we have $G \cap \{x \in \mathcal{H}: x_0=0\}=\{0\}^\infty \subset \omega(L_{n+1})$ hence $(A4)$ holds. As $G \cap \omega(L_{n+1})=\{0\}^\infty \neq \emptyset$, we have that $X_{n+1}/_\sim$ is a circle, which proves $(A3)$. 

    So far, we already have shown that $(X_{n+1},L_{n+1},\varphi_{n+1})$ satisfies $(A1)-(A5)$, $(A7)$ and $(A10)$. We need to show $(A6),(A8)$ and $(A9)$

    \begin{lem}
    \label{lem:A9}
        The triple $(X_{n+1},L_{n+1},\varphi_{n+1})$ satisfies the condition $(A9)$.
    \end{lem}

    \begin{proof}
           First, choose any  natural numbers with $i_0 \geq i_1 \geq \cdots \geq i_{n}$ and  any $y \in \frac12\theta(\varphi_n(P^n_{i_1,\ldots, i_{n+1}}))$. Using $(A10)$ and $(A7)$ for indexes $k<n$ one by one and shifting the coordinates to the right we have that  $y=\frac12(0,t,f(t),\ldots, \frac1{2^{n-2}}g_{i_1,\ldots i_{n-1}}(t),\frac1{2^{n-1}}g_{i_1,\ldots i_{n}}(t),0, \ldots)$ for some $t \in (0,1]$.
           
           By $(A9)$ there is a sequence $t^{(i_0)} \in P^n_{i_0,\ldots,i_n}$ such that $\lim_{i_0 \to \infty} \varphi_n(t^{(i_0)})=(y_0,\ldots,y_{n+1},0,\ldots)$. We wish to show that $ \lim_{i_0 \to \infty}\varphi_{n+1}(t^{(i_0)})=y $. By $(A7)$, all coordinates of  $\varphi_{n+1}(t^{(i_0)})$ coincide with coordinates of $\varphi_{n}(t^{(i_0)})$, aside from the $(n+3)th$ coordinate. It follows that to get  $ \lim_{i_0 \to \infty}\varphi_{n+1}(t^{(i_0)})=y=(y_0,\ldots,y_{n+1},y_{n+2},0,\ldots)$ we only need to check formula for the $(n+3)$th coordinate.

               \begin{equation*}
\begin{split}   
 \lim_{i_0 \to \infty} g_{i_0,\ldots,i_n}(t^{(i_0)})=&\lim_{i_0 \to \infty} (f_{i_n} \circ h_{i_n} \circ \cdots \circ f_{i_0} \circ h_{i_0} \circ f)(t^{(i_0)}) =\\ 
 =&\lim_{i_0 \to \infty} (f_{i_n} \circ h_{i_n} )( g_{i_0,\ldots,i_{n-1}}(t^{(i_0)}))=\\
 =&
 (f_{i_n} \circ h_{i_n} )(\lim_{i_0 \to \infty}  g_{i_0,\ldots,i_{n-1}}(t^{(i_0)}))= \ldots
\end{split}
\end{equation*}
but we have that  $\pi_{n+1}(\varphi_{n}(t^{(i_0)})) \to y_{n+1} $. Using $(A10)$ for  $\pi_{n+1}(\varphi_{n}(t^{(i_0)}))$ and $ y_{n+1}=\pi_n(\varphi_{n-1}(t)) $ we get $g_{i_0,\ldots,i_{n-1}}(t^{(i_0)}) \to g_{i_1,\ldots,i_{n-1}}(t)$, so we can continue the chain of equalities:
\begin{equation*}
    \ldots=(f_{i_n} \circ h_{i_n} )( g_{i_1,\ldots,i_{n-1}}(t)) =g_{i_1,\ldots i_n}(t)=2^ny_{n+2}
\end{equation*}

           so $\lim_{i_0 \to \infty} \pi_{n+2}(\varphi_{n+1} (t^{(i_0)}))= \lim_{i_0 \to \infty} \frac1{2^{n}}g_{i_0,\ldots,i_n}(t^{(i_0)})=\frac1{2^{n}}2^{n}y_{n+2} = y_{n+2} $.
This completes the proof.

    \end{proof}
\begin{lem}
    The equality $\omega(L_{n+1})=\frac12\theta(X_n)$ holds, meaning that the triple $(X_{n+1},L_{n+1},\varphi_{n+1})$ satisfies the condition $(A8)$. 
\end{lem}
\begin{proof}
    Choose $\xi \in  \frac12\theta(X_n)$. Assume first that $\xi =\frac12\theta(\varphi_n(t))\in \frac12\theta(L_n)$ for some $t \in (0,1]$. 
     If $\xi \in \frac12\theta(\varphi_{n}(P^{n}_{i_1,\ldots i_{n},i_{n+1}}))$,the result follows from Lemma~\ref{lem:A9}. Suppose otherwise. This means that $\varphi_{n-1}(t)=\varphi_n(t)$, so  $\xi = \frac12\theta(\varphi_{n-1}(t)) \in \frac12\theta(L_{n-1}) \subset \frac12\theta(X_{n-1})$. By  $(A8)$ in the induction hypothesis, there is a sequence $\varphi_n(\xi^{(k)}) \in L_n$ that converges to $\xi$. If the sequence $\xi^{(k)}$ was chosen to be in the intervals  in $P^{n+1}$, then by Lemma~\ref{lem:A9}, $\xi$ would have to be an element of $\frac12\theta(
    \varphi_{n}(P))$ for some $P \in P^n$, contradicting our assumptions. This means that $\xi^{(k)}$ can be chosen so that $\varphi_n(\xi^{(k)})=\varphi_{n+1}(\xi^{(k)})$, giving us a sequence of points in $L_{n+1}$ converging to $\xi$.

Now, let $\xi \in \frac12 \theta(\omega(L_n))$. Choose any $\epsilon>0$. Then by the definition of the limit set, there is $\xi' \in \frac12\theta(L_n)$ with $d(\xi,\xi')<\epsilon$. By the argumentation above, there is also $\xi'' \in L_{n+1}$ with $d(\xi',\xi'')<\epsilon$, hence $d(\xi,\xi'')<2 \eps$. 

Last, assume $\xi \in \frac12\theta(G) \subset \frac12\theta(X_0)$. Then  by $(A9)$ for $n=1$ there is a sequence $t^{(i)}\in (z^{(i)},y^{(i)})$ with $\varphi_0(t^{(i)}) \to \xi$. As for every $n>0$ the equality $\varphi_n(t^{(i)})=\varphi_0(t^{(i)})$, we get the desired result.
  
\end{proof}
  \begin{lem}
\label{lem:P_go_toL}
 For every     $\xi\in \frac12\theta(L_n)$ there is a sequence $t^{(k)}\in P^0_k$
such that $\varphi_{n+1}(t^{(k)})\to \xi$.
Furthermore, if $Y \subset  \frac12\theta(L_n)$ is subcontinuum such that $Y \cap \frac12\theta(L_n) \neq \emptyset $, then there are intervals $J_k\subset P_k^0$ such that $\lim_{k\to \infty}d_H(Y,\varphi_{n+1}(J_k))=0$,  and:
\begin{enumerate}
    \item if $Y=\frac12\theta(\varphi_n([a,b])$ for  some $0<a<b \leq 1$, then $J_k=[(f|_{P^0_k})^{-1}(a),(f|_{P^0_k})^{-1}(b)]$
    \item if $Y=\frac12\theta(\overline{\varphi_n((0,b])})$ for some $b \in (0,1]$ then $J_k=[y^{(k+1)},(f|_{P^0_k})^{-1}(b)]$
\end{enumerate}
\end{lem}
\begin{proof}

 In principle, the result follows from the use of $(A9)$ for $k<n$, considering the points whose $(k+1)$th coordinate is zero. We can also take any $P^0_{i_0}$ to get
   $$\varphi_{n+1}(P^0_{i_0})=\{(t,f(t),\frac12g_{i_0}(t),\ldots,\frac1{2^n}g_{i_0,\ldots,i_n}(t),0,\ldots), t \in P^0_{i_0}\}$$
      But  $(h_{i_0} \circ f) $ is a homeomorphism on $P^0_{i_0}$  and $(h_{i_0}\circ f)(P^n_{i_0,\ldots,i_n}) \subset P^0_{i_1}$ with $i_1 \le i_0$, so $f(s)=f_{i_1}(s)$ for all $s \in(h_{i_0}\circ f)(P^n_{i_0,\ldots,i_n})  $. Therefore, we can write
    $$\varphi_{n+1}(P^0_{i_0})=\{(f^{-1}|_{P^0_{i_0}}(y),y,\frac12(f(h_{i_0}(y)),\ldots,\frac1{2^n}g_{i_1,\ldots,i_n}(h_{i_0}(y)),0,\ldots), y \in f(P^0_{i_0})\}.$$

Fix any $\xi\in \frac12\theta(L_n)$, say 
$\xi=\frac12(0,2y,(f(2y)),\ldots,\frac1{2^{n-1}}g_{i_1,\ldots,i_n}(2y),0,\ldots)$
for some $y \in (0,\frac12]$.

Note that $f(P^0_i)=[f(y^{(i+1)}),f(z^{(i)})]$
with $f(y^{(i+1)})$ decreasing to $0$ and $f(z^{(i)})$ decreasing to $1/2$.
For large $i$ there is $t^{(i)}\in P_i^0$
such that $f(t^{(i)})=y$. Since $h_{i}$
is an affine homeomorphism for any $i$,
it is also clear that $\lim_{i\to\infty }h_i(y)=2y$. This proves that $\varphi_{n+1}(t^{(i)})\to \xi$.

The proof of the "furthermore" part is analogous. Simply, in the case $(1)$ it is enough to prove convergence at the endpoints of $A$ and extend it to all other points with coordinates in $\pi_0(J_k)$. Details are left to the reader. For the case $(2)$ notice that $\varphi_{n+1}(y^{(k+1)})=(y^{(k+1)},f(y^{(k+1)}),0,\ldots)$, but as we mentioned above, $y^{(k+1)} \to 0$, so $\lim_{k \to \infty}\varphi_{n+1}(y^{(k+1)}) =0^\infty \in \frac12\theta(\omega(L_{n}))$, but  $J_k$ are connected, meaning so is $\lim_{k \to \infty} J_k$. That along with the fact that $\lim_{k \to \infty}(J_k) \cap\frac12\theta(L_n) \neq \emptyset$ gives us the result.

\end{proof}

\begin{lem}
\label{lem:notP_go_to_G}
    We have that $\lim_{k\to \infty}d_H(\varphi_{n+1}([z^{(k+1)},y^{(k)}]),\frac12\theta(G))=0$
\end{lem}
\begin{proof}
  As $[z^{(k+1)},y^{(k)}] $ intersects intervals of $P^0$ only at the endpoints, where we keep the zero coordinates, we have that $ \varphi_{n+1}([z^{(k+1)},y^{(k)}])=\varphi_0([z^{(k+1)},y^{(k)}])$, so the result follows as $\lim_{k \to \infty}f(y^{(k)})=0$ and $\lim_{k \to \infty}f(z^{(k)})=\frac12$, and $\frac12\theta(G)=\{0\}\times[0,\frac12]\times\{0\}^\infty$.
\end{proof}

     \begin{lem}
        The triple $(X_{n+1},L_{n+1},\varphi_{n+1})$ satisfies the condition $(A6)$, 
    \end{lem}

    \begin{proof}
            Fix any subcontinuum nondegenerate $Y \subset \omega(L_{n+1})  $. 
    
    Assume first that $Y \subset \frac12 \theta(L_n)$, it follows that $Y$ is an arc. By Lemma~\ref{lem:P_go_toL} there is a sequence $[a^{(k)},b^{(k)}] \subset P^0_k$ such that $\varphi_{n+1}([a^{(k)},b^{(k)}]) \to Y$ in Hausdorff metric.

    Assume now that $Y \subset \frac12\theta(\overline{L_n})$  and fix $\epsilon>0$.  By the induction hypothesis $(A5)$, we can find $[a,b] \subset (0,1]$ such that  $d(\frac12\theta(\varphi_n([a,b])),Y)<\epsilon$, but then, by the step above, we can find $[a',b'] \in (0,1]$ such that $d_H(\varphi_{n+1}([a',b']),\frac12\theta(\varphi_n([a,b]))) <\epsilon$, so $d_H(\varphi_{n+1}([a',b']),Y)<2\epsilon$.

    Suppose  $Y \subset \frac12\theta(G)$ and fix $\epsilon>0$. Choose $[a^{(k)},b^{(k)}] \subset (z^{(k)},y^{(k)})$ with $\varphi_0([a^{(k)},b^{(k)}]) \to Y$. Since $(z^{(k)},y^{(k)})\cap \bigcup_{P \in P^0} P= \emptyset$, we have that $\varphi_{n+1}([a^{(k)},b^{(k)}])=\varphi_0([a^{(k)},b^{(k)}])$, giving us the desired result.

   We are left to consider the case  $\Inter(Y_1)=\Inter(Y \cap \frac12\theta(G) )\neq \emptyset$ and $\Inter Y_2= \Inter (Y \cap \frac12 \theta(\overline{L_n}) )\neq \emptyset $. If $Y=\omega(L_{n+1})$, then 
the statement is obvious by the definition of limit set.
  Consider first the case  $Y_1 \cap Y_2$ is one point, either $e^0=(0,0,\ldots)$ or $e^1=(0,\frac12,0,\ldots)$ (the endpoints of $\frac12\theta(G)$).

    In the case of $e^0$, fix $\delta^{(k)}_1,\delta^{(k)}_2$ so that  $\varphi_{n+1}([y^{(k+1)},z^{(k)}-\delta^{(k)}_2]) \to Y_1 $  and $\varphi_{n+1}([z^{(k+1)}+\delta_1^{(k)},y^{(k+1)}]) \to Y_2$, using Lemma~\ref{lem:P_go_toL} .

    This gives us $\varphi_{n+1}([z^{(k+1)}+\delta_1^{(k)},z^{(k)}-\delta_2^{(k)}]) \to Y$. 
   As $e_0 \in Y$, by connectedness we have that $\omega(L_{n+1}) \subset Y$, hence $\lim_{k \to \infty}\varphi_{n+1}(y^{(k)}) \in Y$ and so the arc approximating $Y_1$ can be chosen to include $y^{(k+1)}$.
    
    In the case of $e_1$ we act analogously, to get $\varphi_{n+1}([y^{(k+1)}+\delta_1,y^{(k)}-\delta_2]) \to Y$.

    Finally, consider the case that $Y_1\cap Y_2$ contains both $e^0$, $e^1$.
    Then either $Y_1=\frac12 \theta(\overline{L_n})$ or $Y_2= \frac12\theta(G)$ because $Y$ is connected.  Without the loss of generality suppose $\frac12\theta(G) \subset Y$. 
   It follows that $Y_1=\frac12 \theta(\overline{L_n}) \setminus \frac12\theta(\varphi_n([a,b]))$ for some $0<a<b \le1$.  Denote  $a^{(k)}= (f|_{P^0_k})^{-1}(a)$ and $b^{(k)}=(f|_{P^0_k})^{-1}(b)$. By Lemma~\ref{lem:P_go_toL} we have that $\varphi_{n+1}( [b^{(k+1)},z^{(k+1)}]\cup[y^{(k)},a^{(k)}] )\to Y_1$ and by Lemma~\ref{lem:notP_go_to_G} we have that $\varphi_{n+1}([z^{(k+1)},y^{(k)}]) \to \frac12\theta(G)=Y_2$. All together this gives us that $\varphi_{n+1}([b^{(k+1)},a^{(k)}]) \to Y$.
    \end{proof}

    Observe that maps $\varphi_n$ form a Cauchy sequence (in the space of continuous maps $[0,1]\to \mathcal{H}$) and therefore  $\varphi_\infty(t)=\lim_{n \to \infty}\varphi_n(t)$ is a well-defined continuous map. 
 
    Note that $\pi_0(\varphi_\infty(t))=t$, hence is $1-1$. This shows that $L_\infty = \varphi_\infty((0,1])$ is an oscillatory quasi-arc in $X_\infty$. 

    We have $X_\infty \cap \{x\in \mathcal{H} : x_0=0\}=\omega(L_\infty)$, so  $X_\infty /_\sim$ is a circle, the map $\phi=\pi_\sim|_{X_\infty}$ has a unique tranche $\omega(L_\infty)$, so $X_\infty$ is a tranched graph.

    Similarly to the Example~\ref{exmp:infinite_order_quasi_arc}, the continuum $X_\infty$ has a self-similar property, i.e. $\omega(L_\infty)=\frac12\theta(X_\infty)$, meaning $L_\infty$ is an $\infty$-order oscillatory quasi-arc.

    \begin{lem}
        $X_\infty$ is a generalized sin(1/x)-type continuum
    \end{lem}
    \begin{proof}
         Choose any continuum $Y \subset \omega(L_\infty)$, and any $\epsilon>0$. Pick $n$ large enough so that for any $x,y \in \mathcal{H}$ if $x_i=y_i$ for all $i \leq n$ then $d(x,y)<\epsilon$. Then for any subcontinuum $C \subset \mathcal{H}$ and $C_n=\{(x_0,\ldots,x_n,0,\ldots) \in \mathcal{H}: x \in C\}$ we have $d_H(C,C_n)<\epsilon$.
    
    It follows that $d_H(X_n,X_\infty)<\epsilon$. Let $Y_n$ be the projection of $Y$. It follows that $d_H(Y,Y_n)<\epsilon$. As $X_n$ is a generalized sin(1/x)-type continuum, there is an arc $[a,b] \subset (0,1]$ with $d_H(Y_n,\varphi_n([a,b])<\epsilon $, but by the choice of $\epsilon$ we have $d_H(\varphi_n([a,b]),\varphi_\infty([a,b]))<\epsilon$, so the triangle inequality gives us $d_H(\varphi_\infty([a,b]_,Y)<3\epsilon$.
    \end{proof}

    \begin{lem}
        Continuum $X_\infty$ is arcwise connected.
    \end{lem}
    \begin{proof}
            It is enough to show that for every point $x \in X_\infty$ there is an arc from $x$ to $(1,0,0,\ldots) =\varphi_\infty(1)\in X_\infty$. If $x \in G$, then the arc is  $[x_0,1]\times\{0\}^\infty$, if $x \in L_\infty$, then we take $\varphi_\infty([x_0,1])$. Suppose now $x_0=0$, but there is a coordinate $n$ such that $x_n \neq 0$. This means that $x\in \frac1{2^n}\theta^n(L_\infty\cup G)$, so the desired arc is $A=G \cup \frac12\theta(G) \cup\ldots\cup \frac1{2^n}\theta^n(G \cup \varphi_\infty([x_n,1]))$
             or
            $A=G \cup \frac12\theta(G) \cup\ldots\cup \frac1{2^n}\theta^n([x_0,1]\times\{0\}^\infty)$ depending on the location of $x$.
    \end{proof}

\section{Relation to other classes and dynamics of tranched graph maps}\label{sec:dynamics}

In \cite{MR4471558} M.~Mihokov\'{a} studied minimal sets on continua with a free interval, where \textit{free interval} is any space homeomorphic to $\R$. We say that $J$ is a dense free interval in $X$ if $J$ is a free interval and $\overline{J}=X$. In \cite{MR4471558} the objects studied are continua $X$ that can be expressed in the form:
$$ X = L \cup J \cup R$$
where  $L,R$ are  nowhere dense locally connected continua, disjoint from the dense free interval $J$ that can be split into two rays $J_L$ and $J_R$ such that $\omega(J_L)=L$ and $\omega(J_R)=R$. Similarly to \cite{MR4471558}, we will denote by $\mathcal{C}$ the class of such continua $X$.
\begin{lem}
    Suppose $X \in \mathcal{C}$. Then $X$ is a tranched graph.
\end{lem}
\begin{proof}
    Let $Y=X / _\sim$ where $x \sim y$ if and only if $x=y$ or both $x,y \in L$ or both $x,y \in R$. Let $\phi \colon X \ra Y$ the associated quotient map. Since $L$ and $R$ are closed and nowhere dense, $\phi$ is continuous and monotone, and moreover all fibers of $\phi$ are nowhere dense and the only (possible) nondegenerate fibers of $\phi$ are $L$ and $R$. It follows that the set of points with degenerate preimage is dense, so  $\phi$ satisfies Definition~\ref{def:tranched_graph} and so $X$ satisfying the definition from \cite{MR4471558} is a tranched graph.
\end{proof}
Using \cite[Remark 5]{MR4471558}  it is possible  to classify all  graphs $Y$ 
 that can appear in the definition of the tranched graphs in this class:
\begin{rem}
    Let $X=L \cup J \cup R$, $\phi \colon X \mapsto Y$ be mappings from the definition of tranched graph. Then, up to homeomorphism:
    \begin{enumerate}[(a)]
        \item If both $L$ and $R$ are singleton then $X$ has no tranches and
            \begin{itemize}
                \item If $L=R$, then $X=Y$ is a circle,
                \item If $L \neq R$, then $X=Y=[0,1]$
            \end{itemize}
        \item If $L$ is nondegenerate and $R$ is a singleton, then $L$ is the only tranche of $X$ and
            \begin{itemize}
                \item if $L \cap R = \emptyset$, then $Y=[0,1]$ and $\phi^{-1}(0)=L$ is a tranche.
                \item if $L \cap R \neq \emptyset$, then $Y$ is a circle.
            \end{itemize}
        \item if both $L$ and $R$ are nondegenerate
            \begin{itemize}
                \item If $L \cap R = \emptyset$, then $Y=[0,1]$ and $X$ has two tranches: $\phi^{-1}(0)=L$ and $\phi^{-1}(1)=R$
                \item If $L \cap R \neq \emptyset$, then $Y$ is a circle and $X$ has one tranche $L \cup R$
            \end{itemize}
    \end{enumerate}
\end{rem}
 It is easy to check that all examples provided by \cite{MR4471558} satisfy the definition of a generalized sin(1/x)-type continuum. The topological structure of continua with dense free interval is somewhat restricted; still, following examples provided earlier in this paper, it is not hard to show that the elements of the class $\mathcal{C}$
do not have to be generalized sin(1/x)-type continua, neither do they have to be hereditary tranched graphs.

 We say that a map $f \colon X \to X$ is \textit{(topologically) mixing} if for any nonenmpty open subsets $U,V \subset X$ there is a number $N \in \N$ such that for all $n>N$ $f^n(U) \cap V \neq \emptyset$. In the standard hierarchy of chaotic maps, definition of mixing maps can be extended to a stronger definition of topologically exact maps, where a map $f \colon X \ra X$ is \textit{(topologically) exact}  or \textit{locally eventually onto} (leo) if for every open set $U \subset  X$, there exists $N \in \N$ such that $f^N(U)=X$.

In the third chapter of his doctoral thesis (\cite{Drwiega2019}, in Polish) Drwiega studied the dynamics of the sin(1/x) curve, which can be viewed in the framework of this paper as a simple generalized sin(1/x)-type continuum. He showed a lower bound on topological entropy of a continuous topologically transitive  map of a sin(1/x)-continuum, which is $\log3$. He also proved that for a space consisting of two quasi-arcs accumulating on an interval, no mixing map is topologically exact - a result we extend in Theorem~\ref{lem:finite_then_no_LEO} by a different argument. In a more general setting, the dynamics of quasi-graph maps was much more studied (e.g. see \cite{MR3557770},\cite{MR4385436},\cite{MR4580970}) than in the case of generalized sin(1/x)-type continua (see \cite{MR3272777} for some comments). 
The first result that is inspired by the above studies 
is that the set of tranches is invariant under any continuous  onto map. This property imposes essential restrictions on topological and ergodic properties of dynamical systems on these spaces.

\begin{lem}\label{lem:TXinv}
    Suppose $X$ is an 
    tranched graph with a finite set of tranches and  let $\phi \colon X\to Y$ be an associated map from the definition. Then for any continuous surjective mapping $f \colon X \ra X$, set $T_X$ composed of tranches of $X$, i.e. 
    $$
    T_X=\{x\in X : \phi^{-1}(\phi(x))\neq \{x\}\}
    $$
    is $f-$ invariant, meaning $f(T_X)=T_X$.
\end{lem}
\begin{proof}
    Let $X$ be 
    tranched graph  with a finite set of tranches. It follows by Lemma~\ref{lem:finite_tranches}  that all tranches are limit sets of oscillatory quasi-arcs. 
    
    Suppose first that there is $x \in T_X$ such that $f(x)  \not\in  T_X$.  Denote by $L_1, \ldots, L_k$ the  set  of quasi-arcs and assume that $x\in \omega(L_i)$. Suppose that there is $\tilde{x} \in \omega(L_i)$ with $f(\tilde{x}) \neq f(x)$. By continuity, we get that $f(L_i)$ is nondegenerate and arcwise connected, hence contains an oscillatory quasi-arc and $f(\omega(L_i))$ is a non-degenerate limit set. This means that $f(\omega(L_i)) \subset T_X$ and so $f(x) \in T_X$, contradicting the assumptions.  Therefore $\{f(x)\}= f(\omega(L_i))$ for any oscillatory quasi-arcs  $L_i$ such that $x \in \omega(L_i)$,
    in particular $f(\omega(L_i))\cap T_X=\emptyset$.  This means that $f(L_i \cup \omega(L_i))$ is a topological graph and so $f(L_i)$ does not contain an oscillatory quasi-arc. This implies that at least one oscillatory quasi-arc in $X$ does not map onto an oscillatory quasi-arc. But any oscillatory quasi-arc must be an image of an oscillatory quasi-arc, which would contradict  surjectivity. This shows that $f(T_X)\subset T_X$.

    Suppose now that there is $y \in T_X$ such that $y \neq f(x)$ for all $x \in T_X$. This means there is an oscillatory quasi-arc $L \subset X$ that no oscillatory quasi-arc maps to, because otherwise, if $f(K)=L$, then $f(\omega(K))=\omega(L)$, and as a consequence  there is $x \in \omega(K)$ such that $f(x)=y$. 
    But $f$ is surjective, leading to a contradiction.
\end{proof}

Notice that in general we cannot say that set $X \backslash T_X$ is invariant as well, so Lemma~\ref{lem:TXinv} cannot be extended beyond $T_X$.

The following Lemma will be an important tool in the process of describing possible topological dynamics on tranched graphs.

\begin{lem}
\label{lem:tranched_graphs_not_peano}
    Let $X$ be a tranched graph. Then $X$ is a Peano continuum if and only if it is a topological graph.
\end{lem}
\begin{proof}
    Let $X$ be a tranched graph and let $\phi \colon X \ra Y$ be the continuous monotone map from its definition. Assume that $X$ is a Peano continuum but is not a topological graph. It follows that there is at least one tranche $T=\phi^{-1}(y)$ for some $y \in Y$.  Denote by $n$ the valence of the point $y \in Y$. As $T$ is nondegenerate, we can choose $n+1$ distinct points $\{x_0, \ldots, x_n \} \subset T$, denote $\delta=\max d(x_i,x_j) /2$  and let $U_i\subset B(x_i, \delta)$ be open  connected sets. 
    Denote $S=\phi(\bigcup_{i=1}^n U_i)$ and observe that  $S$ contains an $s$-star centered in $y$ for some $s\leq n$, but does not contain $k$-stars for $k>n$. In particular, there is an arc $A \subset Y$ such that $A \subset \phi(U_i)\cap \phi(U_j), j \neq i$. By the density of singleton fibers, there is $y_0 \in A$ with a degenerate preimage. But sets $U_i$ and $U_j$ were disjoint, hence the fiber $\phi^{-1}(y_0)$ cannot intersect both of them, leading to a contradiction.
\end{proof} 
Lemma 3.1 and Corollary 3.2 in \cite{MR3557770} state that for a quasi-graph map, the image of a topological graph does not contain any oscillatory quasi-arc.
The continuous image of a Peano continuum has to be a Peano continuum as well; therefore, Lemma~\ref{lem:tranched_graphs_not_peano} extends the result of \cite{MR3557770} to a more general setting of arcwise connected tranched graphs.
\begin{cor}
\label{cor:image_of_graph_is_graph}
    Let $X$ be an arcwise connected tranched graph and let $G \subset X$ be a topological graph.  For any continuous map $f \colon X \ra X$, the set $f(G)$ does not contain any oscillatory quasi-arcs. In particular, if $f(G)$ is nondegenerate, then $f(G)$ is a topological graph.
\end{cor}

 Now we are ready to show that tranched graphs may support complicated dynamics.

If we glue two Warsaw circles at their tranches (Figure~\ref{fig:mixing_map}) we get an arcwise connected tranched graph, but whose set of singleton fibers is not arcwise connected. It is easy to verify that such continuum doesn't admit a mixing map. On the contrary, the double sided sin(1/x)-continuum (Figure~\ref{fig:double}) is not arcwise connected, but has arcwise connected set of singleton fibers. 
\begin{figure}[ht]

            \centering
            \includegraphics[scale=0.30]{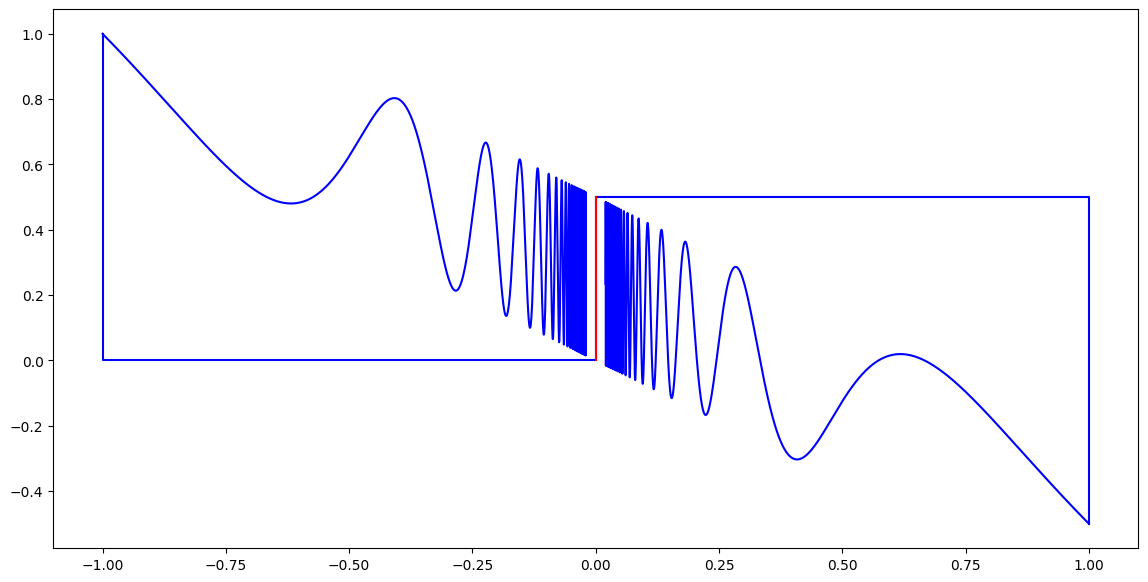}
            \caption{An arcwise connected tranched graph, that doesn't admit a mixing map.  }
            \label{fig:mixing_map}
    \end{figure}

\begin{thm}
    Suppose that $X$ is a tranched graph 
   and let $Y$ be a topological graph such that $\phi \colon X \ra Y$ satisfies the definition.
    Assume that the set $\{x : \phi^{-1}(\phi(x))=\{x\}\}$
    is arcwise connected.
    Then there exists a topologically mixing map $f \colon X \ra X$.
\end{thm}
\begin{proof}

    Let $X$ be a tranched graph  and let $Y$ be a topological graph such that $\phi \colon X \ra Y$ satisfies the definition. Assume that the set $\{x : \phi^{-1}(\phi(x))=\{x\}\}$
    is arcwise connected.
    Removing the images of tranches from $Y$ keeps the space arcwise connected, therefore the number of tranches for $X$ is bounded from above by $b_1(Y)$ (the disconnecting number of $Y$ ) and consequently the continuum $X$ has finitely many tranches. Denote $N= \{y \in Y | \phi^{-1}(y) \text{ is nondegenerate}\}
    =\{y_1, \ldots, y_n\}$.  Let  $Y_1$ be obtained by compactifying $Z_1=Y \backslash \{y_1\}$ in such a way that the set $A_1=Y_1\setminus Z_1\subset \End(Y_1)$. Informally it means that we ,,cut'' the graph $Y$ at $y_1$ 
    obtaining $val(y_1)$ endpoints. Our assumptions guarantee that $Z_1$ and $Y_1$ are connected.
    
    Continue this construction recursively denoting by $Y_k$ the graph obtained by compactifying $Z_k=Y_{k-1} \backslash \{y_k\}$ in such a way that the set $A_k=Y_k\setminus Z_k\subset \End(Y_k)$.
 Denote by $\sim$ the equivalence relation identifying points in each set $A_i$, i.e. $a\sim b$ if $a=b$ or $a,b\in A_i$ for some $i$. Note that $Y_n /_\sim=Y$ up to a homeomorphism. By \cite{MR3255434}, there is a pure mixing map $g \colon Y_n \ra Y_n$ with $\End(Y_n) \subset \Fix(g) $. Moreover, using techniques from the construction of a purely mixing map on the interval (see for example \cite[Chapter 2.2]{MR3616574}), we can get that fixed points of $g$  accumulate on the set of endpoints of $Y_N$ in any prescribed way (i.e. we can control the speed of convergence in the construction). We can pull this map back to $Y$, by setting $h(y)=g(y)$ on points outside $A_i$ and $h(y)=y$ on points in $N$. Finally, we can define the map $f \colon X \ra X$ to be $f(\phi^{-1}(y))=\phi^{-1}(h(y))$ if $y$ is an image of a degenerate fiber and $f(x)=x$ if point $x$ is an element of the tranche of $X$. 
   
    The map is well defined and monotone, and since we may control the convergence of fixed points when defining $g$ we can easily get that $f$ is also continuous. As $g$ was topologically mixing, so is $h$, and as a result is $f$.
\end{proof}
 Recall that a map $f \colon X \ra X$ is topologically exact if for every open set $U \subset  X$, there exists $N \in \N$ such that $f^N(U)=X$. These maps are to some extent excluded on tranched graphs, except for very regular ones, as shown below.

\begin{thm}
\label{lem:finite_then_no_LEO}
    Suppose that $X$ is a tranched graph with a finite and nonempty set of tranches. Then $X$ does not admit a topologically exact map.
\end{thm}
\begin{proof}
    Let $X$ be as in assumptions and suppose $f \colon X \ra X$ is topologically exact. As the set of tranches of $X$ is finite, there is an open connected set that is mapped injectively to a topological graph. As such, there is a topological graph $G \subset X$. As $f$ is topologically exact, there is a natural number $n$ for which $f^n(G)=X$. As $G$ is a topological graph, $X$ is a continuous image of a Peano continuum, so it is a Peano continuum. It is in contradiction with Lemma~\ref{lem:tranched_graphs_not_peano}.
\end{proof}

The assumption that the continuum has finitely many tranches in  
 Theorem~\ref{lem:finite_then_no_LEO} is used to generate a Peano subcontinuum with nonempty interior. If the set of tranches is infinite, this argument cannot be used anymore. The following example shows that for complicated generalized sin(1/x)-type continua, we can get complex dynamics  both in a global (topologically exact map) and local (infinite topological entropy) sense.

\begin{exmp}
    Let $\widehat{X}$ be a continuum constructed in Example~\ref{exmp:dense_no_arcs} and $\sigma \colon \widehat{X} \ra \widehat{X} $ be the left shift, defined as in Lemma~\ref{lem:no_arcs_finite_tranche_subcontinuum}. Then $\sigma$ is topologically exact and has infinite topological entropy.
\end{exmp}
\begin{proof}
    By our construction if $(x_0,x_1, \ldots ) \in \widehat{X}$, then $\sigma(x_0,x_1,\ldots)=(x_1,x_2, \ldots) \in \widehat{X}$, so map $\sigma$ is well-defined on $\widehat{X}$.
       Now choose open set $U \subset \widehat{X}$. By Lemma~\ref{lem:tranche:count}, set of tranches of $\widehat{X}$ is dense in $\widehat{X}$, so there is a tranche $T=\hat{\phi}^{-1}(y) \subset U$. 
       By Lemma~\ref{lem:homeo}, $T$ is homeomorphic to $\widehat{X}$ and  for some $n \in \N$ we have $\sigma^n(T)=\widehat{X}$, as $T \subset U$ it follows that $\sigma^n(U)=\widehat{X}$, so map $\sigma \colon \widehat{X} \ra \widehat{X} $ is topologically exact.
       
 Pick two tranches $T_1,T_2$ of $\widehat{X}$. By Lemma~\ref{lem:homeo} they are homeomorphic to $\widehat{X}$, so for any sequence $s \in \{0,1\}^\N$ there is a point $x_s \in \widehat{X}$ such that $f^{i}(x) \in T_{s(i)}$ for all $i \in \N$. As the topological entropy of the full shift on two symbols is $\log 2$ we have that $h_{top}(\sigma) \geq \log 2$. We can repeat this reasoning by picking $k$ tranches to get that $h_{top}(\sigma) \geq \log k$. Freedom of choice of $k$ gives us that $h_{top}(\sigma) = \infty$.
\end{proof}
\section*{Acknowledgements}

Authors are grateful to Chris Mouron, Logan Hoehn, Udayan Darji, Iztok Bani\v{c},  Pawe\l{} Krupski, Veronica Martinez de la Vega, Jernej \v{C}in\v{c}, Ľubomír Snoha and Tomasz Downarowicz for numerous discussions on the history of tranches in topology and implications of tranche-like objects present in the space, on admissible dynamics on that space.

\bibliographystyle{plain} % We choose the "plain" reference style
\bibliography{refs} % Entries are in the refs.bib file

\end{document}